\documentclass[11pt]{article}

\title{Optimal sub-Gaussian variance proxy for 3-mass distributions}
\date{}

\usepackage{tikz}
\usetikzlibrary{arrows.meta,intersections}
\usepackage{amsmath}
\usepackage{amssymb}
\usepackage{amsthm}
\usepackage{subfigure}
\let\oldsubfigure\subfigure
\renewcommand{\subfigure}[2][]{%
    \oldsubfigure[#1]{%
        \begin{tabular}[b]{@{}c@{}}
            #2
        \end{tabular}%
    }%
}
\usepackage{float}     
\usepackage{graphicx}
\usepackage{subfiles}
\usepackage{natbib}
\usepackage[colorlinks,citecolor=blue,linkcolor=blue,urlcolor=blue]{hyperref}
\usepackage{cleveref}   
\usepackage{times}
\usepackage[margin=2.2cm]{geometry}
\usepackage{comment}
\usepackage{wrapfig}
\usepackage{aliascnt}
\usepackage{stmaryrd}
\usepackage{xcolor}

\tikzset{draft/.style={draw=none}}
\usepackage{pgfplots}
\pgfplotsset{compat=1.18}
\usepgfplotslibrary{fillbetween}

\definecolor{C0}{HTML}{3182ce}
\definecolor{darkred}{rgb}{0.55,0.0,0.0}

\crefalias{section}{appendix}
\crefname{appendix}{Appendix}{Appendices}
\Crefname{appendix}{Appendix}{Appendices}

\newcommand{\doi}[1]{\href{https://doi.org/#1}{\texttt{https://doi.org/#1}}}
\renewcommand{\P}{\mathbb{P}}
\newcommand{\E}{\mathbb{E}}
\newcommand{\arcosh}{\mathrm{arcosh}}

\newcommand{\beq}{\begin{equation}}
\newcommand{\eeq}{\end{equation}}
\newcommand{\bea}{\begin{eqnarray}}
\newcommand{\eea}{\end{eqnarray}}
\newcommand{\beqq}{\begin{equation*}}
\newcommand{\eeqq}{\end{equation*}}
\newcommand{\beaa}{\begin{eqnarray*}}
\newcommand{\eeaa}{\end{eqnarray*}}
\def\beaq{\begin{eqnarray}}
\def\eeaq{\end{eqnarray}}

\newcommand{\sectionref}[1]{Section~\ref{#1}}
\newcommand{\appendixref}[1]{Appendix~\ref{#1}}

\newtheorem{theorem}{Theorem}[section]
\newtheorem{lemma}[theorem]{Lemma}
\newtheorem{proposition}[theorem]{Proposition}
\newtheorem{corollary}[theorem]{Corollary}
\newtheorem{definition}[theorem]{Definition}
\newtheorem{remark}[theorem]{Remark}

\newenvironment{proofof}[1]{\begin{proof}[Proof of #1]}{\end{proof}}

\def\br{\begin{remark}\rm\small}
\def\er{\end{remark}}
\def\bt{\begin{theorem}}
\def\et{\end{theorem}}
\def\bd{\begin{definition}}
\def\ed{\end{definition}}
\def\bp{\begin{proposition}}
\def\ep{\end{proposition}}
\def\bl{\begin{lemma}}
\def\el{\end{lemma}}
\def\bc{\begin{corollary}}
\def\ec{\end{corollary}}

\author{
Soufiane Atouani\\
Université Grenoble Alpes, Inria, CNRS, Grenoble INP, LJK, 38000 Grenoble, France\\
\texttt{soufiane.atouani@inria.fr}
\and
Olivier Marchal\\
Université Jean Monnet Saint-Étienne, CNRS, Institut Camille Jordan UMR 5208,\\
Institut Universitaire de France, Les Forges 2, 20 Rue du Dr Annino, 42000 Saint-Étienne, France\\
\texttt{olivier.marchal@univ-st-etienne.fr}
\and
Julyan Arbel\\
Université Grenoble Alpes, Inria, CNRS, Grenoble INP, LJK, 38000 Grenoble, France\\
\texttt{julyan.arbel@inria.fr}
}
\begin{document}

\maketitle

\begin{abstract}
We investigate the problem of characterizing the optimal variance proxy for sub-Gaussian random variables,whose moment-generating function exhibits bounded growth at infinity. 
We apply a general characterization method to discrete random variables with equally spaced atoms. We thoroughly study 3-mass distributions, thereby generalizing the well-studied Bernoulli case. 
We also prove that the discrete uniform distribution over $N$ points is strictly sub-Gaussian. 
Finally, we provide an open-source Python package that combines analytical and numerical approaches to compute optimal sub-Gaussian variance proxies across a wide range of distributions.
\end{abstract}

\section{Introduction}\label{s:1}
The sub-Gaussian property, first characterized by \citet{Kahane1960PropritsLD} and \citet{buldygin1980sub}, has become a critical tool for understanding the tail behavior of random variables. Since these pioneering works, this property has emerged as a fundamental concept in probability theory due to its profound implications in various mathematical disciplines, such as concentration inequalities \citep{Hoeffding1963, boucheron2003concentration, raginsky2013concentration} and Bayesian statistics \citep{catoni2007pac}. In machine learning, sub-Gaussian tails play a crucial role in bandit algorithms \citep{bubeck2012regret}, in the study of the singular values of random matrices \citep{rudelson2010non}, and in Bayesian neural networks \citep{vladimirova2019understanding, vladimirova2020sub}.

\begin{definition}[Sub-Gaussian variables] A random variable \( X \) with finite mean \( \mu = \mathbb{E}[X] \) is \textit{sub-Gaussian} if there exists a constant \( \sigma^2 > 0 \) such that $\mathbb{E}[\exp(\lambda (X - \mu))] \leq \exp(\lambda^2 \sigma^2/2)$ for all $\lambda \in \mathbb{R}$.
% \begin{equation}
%     \mathbb{E}[\exp(\lambda (X - \mu))] \leq \exp\left(\frac{\lambda^2 \sigma^2}{2}\right), \text{ for all } \lambda \in \mathbb{R}.
% \end{equation}
Such a constant \( \sigma^2 \) is called a \textit{variance proxy}, and we say that \( X \) is \(\sigma^2\)-sub-Gaussian. The \textit{optimal variance proxy} is $\sigma^2_{\mathrm{opt}}(X) = \inf \left\{ \sigma^2 > 0 \text{ such that } X \text{ is } \sigma^2\text{-sub-Gaussian} \right\}$. 
% \[
% \sigma^2_{\mathrm{opt}}(X) = \inf \left\{ \sigma^2 > 0 \text{ such that } X \text{ is } \sigma^2\text{-sub-Gaussian} \right\}.
% \]
A variance proxy is always lower bounded by the variance, as shown by a Taylor expansion of the moment-generating function. When $\sigma^2_{\mathrm{opt}}(X) = \mathrm{Var}[X]$, \(X\) is called \textit{strictly} sub-Gaussian.
\end{definition}

Extensive research on optimal variance proxy has focused on continuous distributions such as Beta and Dirichlet distributions \citep{marchal2017sub}, other bounded support distributions such as Kumaraswamy and triangular distributions \citep{arbel2020strict}, as well as truncated Gaussian and exponential distributions \citep{barreto2024optimal}. 
Despite the prevalence of discrete distributions in modeling count data, binary outcomes, and combinatorial stochastic processes, their sub-Gaussianity remains largely underexplored. 
The first known result on discrete distributions covers the Bernoulli distribution, the simplest discrete distribution, supported on two points, or atoms, 1 and 0, with masses $p$ and $1-p$. \citet{Kearns1998LargeDM} derived the following "exquisitely delicate inequality", quoting \citet{berend2013concentration}, for the Bernoulli moment-generating function:
\begin{equation}\label{eq:KS}
    (1-p)e^{-\lambda p}+pe^{\lambda(1-p)}\leq\exp\left(\frac{1-2p}{4\ln((1-p)/p)}\lambda^2\right),\quad p\in[0,1],\,\,\lambda\in\mathbb{R},
\end{equation}
which is tight and thus implies an optimal variance proxy of $\frac{1-2p}{2\ln((1-p)/p)}$. This result also provides the optimal variance proxy for the binomial distribution which is written as an i.i.d. sum of Bernoulli random variables. 
A natural generalization of the Bernoulli to more than two atoms is the categorical distribution with $N$ atoms $x_i$, $i\in \{1,\ldots,N\}$, and weights $P(X=x_i)=p_i>0$, with $p_1+\cdots+p_N=1$. While a general treatment of $N$-mass categorical distributions seems intractable, we address the case of the 3-mass distribution when atoms are equally spaced. 
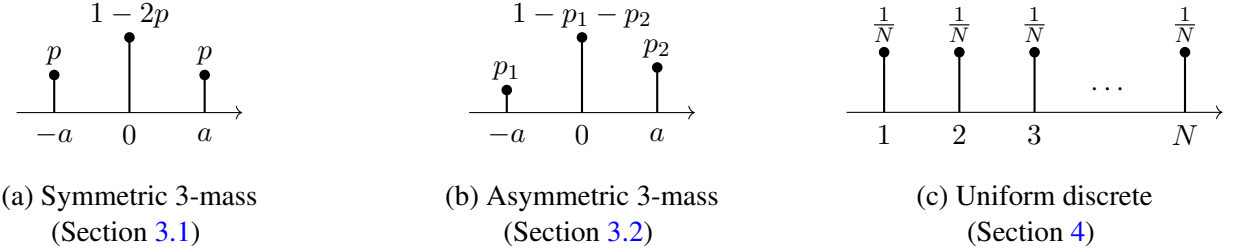
\begin{figure}[tbp]
    \centering
    % (a) Symmetric 3-mass
    \begin{minipage}[b]{0.3\textwidth}
        \centering
        \begin{tikzpicture}[scale=1]
            \draw[->] (-1.5,0) -- (1.5,0);
            \draw[thick] (-1,0) -- (-1,0.5);
            \draw[thick] (0,0) -- (0,1);
            \draw[thick] (1,0) -- (1,0.5);
            \fill (-1,0.5) circle (2pt) node[above] {$p$};
            \fill (0,1) circle (2pt) node[above] {$1 - 2p$};
            \fill (1,0.5) circle (2pt) node[above] {$p$};
            \node at (-1,-0.3) {$-a$};
            \node at (0,-0.3) {$0$};
            \node at (1,-0.3) {$a$};
        \end{tikzpicture}
        
        \vspace{0.2cm}
        (a) Symmetric 3-mass\\
        (\Cref{subsec;symmthreepoints})
        \label{fig:symmetric}
    \end{minipage}
    \hfill
    % (b) Asymmetric 3-mass
    \begin{minipage}[b]{0.3\textwidth}
        \centering
        \begin{tikzpicture}[scale=1]
            \draw[->] (-1.5,0) -- (1.5,0);
            \draw[thick] (-1,0) -- (-1,0.3);
            \draw[thick] (0,0) -- (0,1);
            \draw[thick] (1,0) -- (1,0.6);
            \fill (-1,0.3) circle (2pt) node[above] {$p_1$};
            \fill (0,1) circle (2pt) node[above] {$1 - p_1 - p_2$};
            \fill (1,0.6) circle (2pt) node[above] {$p_2$};
            \node at (-1,-0.3) {$-a$};
            \node at (0,-0.3) {$0$};
            \node at (1,-0.3) {$a$};
        \end{tikzpicture}
        
        \vspace{0.2cm}
        (b) Asymmetric 3-mass\\
        (\Cref{subsec;asymmthreepoints})
        \label{fig:asymmetric}
    \end{minipage}
    \hfill
    % (c) Uniform discrete
    \begin{minipage}[b]{0.3\textwidth}
        \centering
        \begin{tikzpicture}[scale=1]
            \draw[->] (0.5,0) -- (5.5,0);
            \foreach \x in {1,2,3} {
                \draw[thick] (\x,0) -- (\x,0.8);
                \fill (\x,0.8) circle (2pt) node[above] {$\frac{1}{N}$};
                \node at (\x,-0.3) {$\x$};
            }
            \node at (4,0.3) {$\cdots$};
            \draw[thick] (5,0) -- (5,0.8);
            \fill (5,0.8) circle (2pt) node[above] {$\frac{1}{N}$};
            \node at (5,-0.3) {$N$};
        \end{tikzpicture}
        
        \vspace{0.2cm}
        (c) Uniform discrete\\
        (\sectionref{sec:uniform})
        \label{fig:uniform}
    \end{minipage}
    \caption{Probability mass function of the discrete distributions covered in the paper.}
    \label{fig:distrib-plots}
\end{figure}
\paragraph{Contributions and outline.}
We first provide in \sectionref{sec:charac} a characterization of the optimal sub-Gaussian variance proxy for random variables with bounded moment-generating functions, following a general methodology based on function variation analysis. This characterization yields a practical computational procedure via critical points identification and equation solving, enabling explicit computation of the optimal variance proxy (Theorem~\ref{GeneralCharac}). These results are made even more precise when the number of critical points is at most two (Proposition~\ref{Prop01Zeros} and Proposition~\ref{LemmaMerged}). We then apply this approach with a focus on discrete distributions. We start in \sectionref{sec:3-mass} with 3-point distributions, both symmetric and asymmetric, extending the classical Bernoulli case. In the symmetric setting,  Theorem~\ref{TheoSymmetric3mass} uncovers a phase transition: for probabilities $p \geq \frac{1}{6}$, strict sub-Gaussianity holds, whereas for $p < \frac{1}{6}$, it does not, and we derive an explicit characterization of the optimal proxy through a pair of solvable equations. In the asymmetric case, we delineate two regimes depending on the relationship between the central mass and the edge probabilities, and in one of these, we provide a closed-form expression for the optimal proxy (Theorem~\ref{TheoRegime1}). We establish in \sectionref{sec:uniform} that discrete uniform distribution over $N$ equally spaced points is strictly sub-Gaussian for all $N\geq 2$, using a moment-based analysis of the exponential family induced by the log-partition function (Theorem~\ref{TheoRDisceteDistribution}). 
Finally, we describe in  \sectionref{sec:software} the computational framework we developed to support reproducibility, providing an open-source Python package\footnote{\label{note1}The package is  available at \url{https://github.com/jarbel/sub-Gaussian-implementation.git} with comprehensive documentation, installation instructions, and usage examples.} that combines analytical and numerical approaches to compute optimal sub-Gaussian variance proxies across a wide range of distributions.
\section{Characterization of optimal sub-Gaussian variance proxy}\label{sec:charac}
\sloppy{Let $Y$ be a real random variable with finite moments and \( \mu = \mathbb{E}[Y] \). Denote by $M_Y(\lambda):=\ln \left(\mathbb{E}[\exp(\lambda (Y - \mu))]\right)$ the cumulant-generating function of the centered random variable $Y-\mu$. In this paper, we shall always assume that the random variable $Y$ is such that $M_Y$ is a smooth function}. Define 
\beq
\label{eq:g_Y}
g_Y(\lambda;\sigma^2):=\frac{1}{2}\lambda^2\sigma^2- M_Y(\lambda)= \frac{1}{2}\lambda^2\sigma^2-\ln \left(\mathbb{E}[\exp(\lambda (Y - \mu))]\right),
\eeq 
which is a smooth function of $(\lambda,\sigma^2)$. We have $g'_Y(\lambda;\sigma^2):= \lambda\sigma^2- M_Y'(\lambda) \, \text{ and } \, g_Y''(\lambda,\sigma^2)=\sigma^2-M_Y''(\lambda)$. 
Observe that $g_Y(0,\sigma^2)=g_Y'(0,\sigma^2)=0$. Moreover, for $\lambda\neq 0$, the equation $g_Y'(\lambda,\sigma^2)=0$ is equivalent to $\sigma^2=M_Y'(\lambda)/\lambda $. Thus, the system of equations $g_Y(\lambda;\sigma^2)=0$ and $g_Y'(\lambda,\sigma^2)=0$ with $\lambda\neq 0$ is equivalent to  
\beq \label{eq:sys-eq}
\sigma^2=\frac{1}{\lambda} M_Y'(\lambda) \text{ and } \lambda\neq 0 \text{ solution of  } \lambda M_Y'(\lambda)- 2M_Y(\lambda)=0.\eeq
Define the following sets:
\begin{align*}
\mathcal{L}_c^*&:=\Big\{\lambda_c \in \mathbb{R}^* \;\Big|\;  \lambda_c M_Y'(\lambda_c)- 2M_Y(\lambda_c)=0 \,\, \text{ and }\lambda_c \text{ local minimum of } g_Y\Big(.;\sigma^2=\frac{ M_Y'(\lambda_c)}{\lambda_c}\Big) \Big\},\\
\mathcal{S}_c^*&:=\Big\{\frac{M_Y'(\lambda_c)}{\lambda_c}  \;\Big|\; \lambda_c\in \mathcal{L}_c^*\Big\}.
\end{align*}
We complement the two previous sets by defining $\mathcal{L}_c:=\mathcal{L}_c^*\cup\{0\}\,\,, \,\mathcal{S}_c:=\mathcal{S}_c^*\cup\{\operatorname{Var}[Y]\}$. Let us first make the following observation.
\begin{proposition}[Asymptotics at infinity and sufficient condition for sub-Gaussianity.]\label{AsymptInf}
\sloppy{If the cumulant-generating function $M_Y$ is a smooth function and satisfies $M_Y(\lambda)\overset{\lambda\to\pm \infty}{=}o(\lambda^2)$, then $Y$ is \emph{sub-Gaussian} and $\lim_{\lambda\to \pm\infty}g_Y(\lambda;\sigma^2)=+\infty\,,\, \lim_{\lambda\to \pm\infty}g'_Y(\lambda;\sigma^2)=\pm\infty\,,\, \lim_{\lambda\to \pm\infty}g''_Y(\lambda;\sigma^2)=\sigma^2$.}
% \beq \lim_{\lambda\to \pm\infty}g_Y(\lambda;\sigma^2)=+\infty\,,\, \lim_{\lambda\to \pm\infty}g'_Y(\lambda;\sigma^2)=\pm\infty\,,\, \lim_{\lambda\to \pm\infty}g''_Y(\lambda;\sigma^2)=\sigma^2\eeq
% Moreover a sufficient condition is that $Y$ is a bounded random variable.    
\end{proposition}

%\begin{comment}
    
\begin{figure}[htbp]
  \centering
  \scalebox{.8}{
  \usetikzlibrary{arrows.meta}
%\usepackage{pgfplots}
%\pgfplotsset{compat=1.18}
%\usepackage{amsmath,xcolor}

%\begin{document}
\centering
\begin{tikzpicture}

%====================== PARAMÈTRES & FONCTIONS ======================
\pgfmathsetmacro{\pone}{0.05}
\pgfmathsetmacro{\ptwo}{0.01}
\pgfmathsetmacro{\m}{\ptwo-\pone}

\pgfmathdeclarefunction{uzero}{1}{\pgfmathparse{\pone*exp(-#1)+\ptwo*exp(#1)+1-\pone-\ptwo}}
\pgfmathdeclarefunction{duzero}{1}{\pgfmathparse{-\pone*exp(-#1)+\ptwo*exp(#1)}}
\pgfmathdeclarefunction{ntilde}{1}{\pgfmathparse{#1*(duzero(#1)/uzero(#1)) - 2*ln(uzero(#1))}}
\pgfmathdeclarefunction{gfun}{2}{\pgfmathparse{0.5*#2*#1*#1 + \m*#1 - ln(uzero(#1))}}

% Racines
\pgfmathsetmacro{\lamone}{-5.4108345267}
\pgfmathsetmacro{\lamtwo}{-1.8174215932}
\pgfmathsetmacro{\lamthr}{ 9.0943387806}

% Sigmas critiques
\pgfmathsetmacro{\scone}{(duzero(\lamone)/uzero(\lamone) - \m)/(\lamone)}
\pgfmathsetmacro{\sctwo}{(duzero(\lamtwo)/uzero(\lamtwo) - \m)/(\lamtwo)}
\pgfmathsetmacro{\scthr}{(duzero(\lamthr)/uzero(\lamthr) - \m)/(\lamthr)}

% Valeurs de g aux points
\pgfmathsetmacro{\gA}{gfun(\lamone,\scone)}
\pgfmathsetmacro{\gB}{gfun(\lamtwo,\sctwo)}
\pgfmathsetmacro{\gC}{gfun(\lamthr,\scthr)}

%========================= AXE PRINCIPAL ============================
\begin{axis}[
  name=main, clip=false,
  width=16cm, height=10cm,
  xlabel={$\,\lambda$},
  axis lines=middle,
  xmin=-6, xmax=10,
  ymin=-0.3, ymax=1.6,
  domain=-6:10, samples=800,
  enlarge x limits={abs=0.6}, enlarge y limits=0.05
]
  \addplot[smooth, solid, line width=1.2pt, draw=black, restrict y to domain=-0.3:1.7] ({x},{ntilde(x)});
  \addplot[domain=-6:10.6, samples=2, gray]{0};

  \addplot[only marks, mark=*, mark size=1.7pt] coordinates {(\lamone,0) (\lamtwo,0) (\lamthr,0)};
  \node[above right] at (axis cs:\lamone-0.65,0) {$\lambda_{c_1}$};
  \node[above right] at (axis cs:\lamtwo-0.4,0) {$\lambda^{\ast}$};
  \node[above left ] at (axis cs:\lamthr,0) {$\lambda_{c_2}$};

  % Coordonnées utiles
  \path (axis cs:\lamone,0) coordinate (Z1);
  \path (axis cs:\lamtwo,0) coordinate (Z2);
  \path (axis cs:\lamthr,0) coordinate (Z3);
  \path (axis cs:0,0)      coordinate (O);     % <-- origine pour "Var ↔ 0"

  % Boîtes bleues — on leur donne des NOMS pour y connecter des traits
  \node[draw=green!60!black, rounded corners, line width=1pt, fill=white, font=\scriptsize, anchor=east]
    (sc1box) at (axis cs:-3.8,-0.255) {$\sigma_{\mathrm{opt}}^2=s_{c_1}\approx 0.17$};
  \node[draw=blue!70, rounded corners, line width=1pt, fill=white, font=\scriptsize, anchor=east]
    (varbox) at (axis cs:2.2,-0.255) {$\mathrm{Var}\approx 0.059$};
  \node[draw=blue!70, rounded corners, line width=1pt, fill=white, font=\scriptsize, anchor=west]
    (sc2box) at (axis cs:7.4,-0.255) {$s_{c_2}\approx 0.11$};

  %--- Liens pointillés demandés ---
  \draw[green!60!black, densely dashed, line width=.8pt, shorten <=1mm]
    (Z1) -- (sc1box.north);   % λ_c1 ↔ s_c1
  \draw[blue!70, densely dashed, line width=.8pt, shorten <=1mm]
    (O)  -- (varbox.north);   % 0 ↔ Var
  \draw[blue!70, densely dashed, line width=.8pt, shorten <=1mm]
    (Z3) -- (sc2box.north);   % λ_c2 ↔ s_c2
\end{axis}

%=========================== ZOOM 1 (λc1) — BOX =====================
\begin{axis}[
  name=inset1, at={(main.north west)}, xshift=8mm, yshift=0mm, anchor=north west,
  width=5.2cm, height=3.8cm,
  xmin=\lamone-0.45, xmax=\lamone+0.45,
  ymin=\gA,              ymax=\gA+0.03,
  domain=\lamone-0.45:\lamone+0.45, samples=600,
  axis lines=box, axis line style={draw=black},
  axis background/.style={fill=white},
  scaled x ticks=false, xtick distance=0.3,
  xticklabel style={/pgf/number format/fixed,/pgf/number format/fixed zerofill,/pgf/number format/precision=1},
  ytick=\empty, yminorticks=false, minor tick num=0,
]
  \addplot[green!60!black, line width=1pt, smooth] ({x},{gfun(x,\scone)});
  \addplot[blue, densely dashed, line width=.9pt]
    coordinates {(\lamone,\gA) (\lamone,\gA+0.03)};
\end{axis}

%=========================== ZOOM 2 (λ*) — BOX ======================
\begin{axis}[
  name=inset2, at={(main.north)}, yshift=0mm, xshift=3mm, anchor=north,
  width=5.2cm, height=3.8cm,
  xmin=\lamtwo-0.45, xmax=\lamtwo+0.45,
  ymin=\gB-0.01,    ymax=\gB,
  domain=\lamtwo-0.45:\lamtwo+0.45, samples=600,
  axis lines=box, axis line style={draw=black},
  axis background/.style={fill=white},
  scaled x ticks=false, xtick distance=0.3,
  xticklabel style={/pgf/number format/fixed,/pgf/number format/fixed zerofill,/pgf/number format/precision=1},
  ytick=\empty, yminorticks=false, minor tick num=0,
]
  \addplot[red!80!black, line width=1pt, smooth] ({x},{gfun(x,\sctwo)});
  \addplot[blue, densely dashed, line width=.9pt]
    coordinates {(\lamtwo,\gB-0.02) (\lamtwo,\gB+0.02)};
\end{axis}

%=========================== ZOOM 3 (λc2) — BOX =====================
\begin{axis}[
  name=inset3, at={(main.north east)}, xshift=+16mm, yshift=0mm, anchor=north east,
  width=5.2cm, height=3.8cm,
  xmin=\lamthr-0.45, xmax=\lamthr+0.45,
  ymin=\gC, ymax=\gC+0.03,
  domain=\lamthr-0.45:\lamthr+0.45, samples=600,
  axis lines=box, axis line style={draw=black},
  axis background/.style={fill=white},
  scaled x ticks=false, xtick distance=0.3,
  xticklabel style={/pgf/number format/fixed,/pgf/number format/fixed zerofill,/pgf/number format/precision=1},
  ytick=\empty, yminorticks=false, minor tick num=0,
]
  \addplot[green!60!black, line width=1pt, smooth] ({x},{gfun(x,\scthr)});
  \addplot[blue, densely dashed, line width=.9pt]
    coordinates {(\lamthr,\gC-0.03) (\lamthr,\gC+0.03)};
\end{axis}

%=========================== CONNECTEURS VERS LES BOÎTES =============
\draw[green!60!black, line width=1pt, -{Stealth[length=2.4mm]}, shorten >=1mm]
  (Z1) -- ([yshift=-3mm]inset1.south); 
\draw[red!60!black, line width=1pt, -{Stealth[length=2.4mm]}, shorten >=1mm]
  (Z2) -- ([yshift=-3mm]inset2.south); 
\draw[green!60!black, line width=1pt, -{Stealth[length=2.4mm]}, shorten >=1mm]
  (Z3) -- ([yshift=-2mm]inset3.south);

\end{tikzpicture}
%\end{document}
  }
  \caption{Illustration of \Cref{GeneralCharac} in the case of an asymmetric $3$-mass distribution $Y$ with parameters $p_1=0.05$ and $p_2=0.01$ (see~\Cref{subsec;asymmthreepoints}). The black curve represents function $\lambda\mapsto \lambda M_Y'(\lambda)-2M_Y(\lambda)$ of Equation~\eqref{eq:sys-eq}. The green/red box plots represent the local behavior of $\lambda\mapsto g_Y\left(\lambda;\sigma^2=\frac{M_Y'(\lambda^*)}{\lambda^*}\right)$ at each zero $\lambda^*$ of the black curve to decide if $\lambda^*$ is a local minimum of $g_Y$ (in green) or not (in red). In this example, $\mathcal{L}_c^*=\{\lambda_{c_1}\approx -5.41,\lambda_{c_2}\approx 9.09\}$ yielding $\mathcal{S}_c^*=\{s_{c_1}\approx 0.17$ and $s_{c_2}\approx 0.11\}$ while $\operatorname{Var}[Y]\approx 0.059$. Optimal variance proxy is thus $\sigma_{\mathrm{opt}}^2=s_{c_1}\approx 0.17$. 
  }
  \label{fig:/figures/ImageTheorem2}
\end{figure}
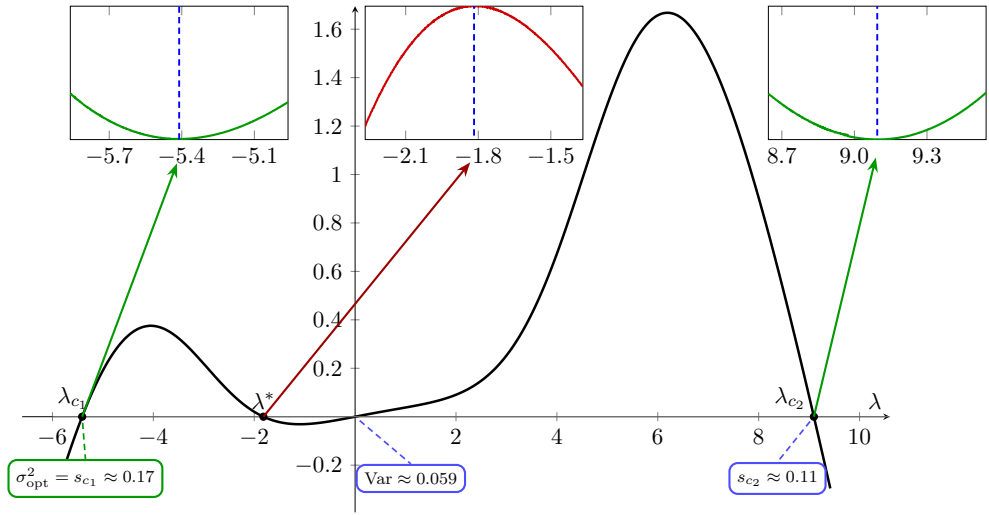
%\end{comment}

Our first main theoretical result is the following theorem that uses the sets $\mathcal{S}_c^*$ to characterize the optimal variance proxy, as illustrated in \Cref{fig:/figures/ImageTheorem2}.

\begin{theorem}[Characterization of the optimal variance proxy.]\label{GeneralCharac} Assume that $M_Y$ is a smooth function and that $M_Y(\lambda)\overset{\lambda\to \pm\infty}{=}o(\lambda^2)$. Then, the optimal variance proxy is characterized by
\beqq \sigma_{\mathrm{opt}}^2 = \max\{ \operatorname{Var}[Y],\; \sup\mathcal{S}_c^* \}. \eeqq   
\end{theorem}

\begin{remark}The main advantage of the characterization of the optimal variance proxy by \Cref{GeneralCharac} is that it provides a numerical way to obtain the optimal variance proxy in practice. Indeed, in order to determine it, one can numerically solve for the equation $\lambda M_Y'(\lambda)-2M_Y(\lambda)=0$. When numerical solutions are found, one should check numerically if $\lambda$ is a local minimum of $g_Y\left(.;\sigma^2=M_Y'(\lambda)/\lambda\right)$ (a sufficient condition being that $g_Y''\left(\lambda;\sigma^2=M_Y'(\lambda)/\lambda\right)>0$). Collecting all solutions, one then selects the optimal variance proxy by looking at the maximal values of $M_Y'(\lambda)/\lambda$ that can easily be computed numerically. In practice, such an algorithm is particularly efficient if the number of solutions of $\lambda M_Y'(\lambda)-2M_Y(\lambda)=0$ is low, and if one can bound the intervals on which to look for numerical solutions.  
\end{remark}

\begin{remark}
In practice, for a given family of distributions, one should study the elements of $\mathcal{L}_c^*$. If $\mathcal{L}_c^*$ is non-empty, one should compare its elements with the corresponding values of $s_c$ to select the optimal variance proxy. This strategy is particularly efficient when $\mathcal{L}_c^*$ contains very few elements or when these elements can be explicitly expressed in closed-forms.
\end{remark}

The main numerical and theoretical difficulty to study elements of $\mathcal{L}_c^*$ is the fact that they must be local minima of $g_Y(.;\sigma^2)$. This property is tricky to verify because the second derivative may also vanish at these points making the analysis complicated. However, when $g_Y''(.;\sigma^2)$ has very few zeros, it is possible to study its sign and thus remove this complication.  

\begin{proposition}[Case when $g_Y''$ has at most one zero on a half-line.]\label{Prop01Zeros} Assume that $M_Y$ is a smooth function and that $M_Y(\lambda)\overset{\lambda\to \pm\infty}{=}o(\lambda^2)$. Assume that $g_Y''(.;\sigma^2)$ has no or one zero on $\mathbb{R}_+^*$ for some $\sigma^2\geq \operatorname{Var}[Y]$, then $g_Y(.;\sigma^2)$ is positive on $\mathbb{R}_+^*$. A similar result holds for $\mathbb{R}_-^*$.
\end{proposition}

Unfortunately the situation becomes more involved when $g_Y''(.;\sigma^2)$ has more than one zero on a half-line. However, we can still get information when $g_Y''(.;\sigma^2)$ has exactly two zeros on a half-line. Two cases are studied in \Cref{LemmaNumber} and \Cref{LemmaNumber2} provided in \appendixref{appendix:tech-lemmas}, which can be put together to obtain the following proposition.

\begin{proposition}[Case when $g_Y''$ has exactly two zeros on a half-line.]\label{LemmaMerged}Assume that $M_Y$ is a smooth function and that $M_Y(\lambda)\overset{\lambda\to \pm\infty}{=}o(\lambda^2)$. Also assume that for any $\sigma^2>0$, $g_Y''(.;\sigma^2)$ has exactly two zeros $(\lambda_1(\sigma),\lambda_2(\sigma))$ such that $0<\lambda_1(\sigma)<\lambda_2(\sigma)$. Then, with a similar result on $\mathbb{R}_-$:
\begin{itemize}
    \item The equations $g_Y(\lambda,\sigma^2)=0=g_Y'(\lambda,\sigma^2)$ have at most one solution on $\mathbb{R}_+^*\times  \mathbb{R}_+^*$. When it exists, this unique solution $(\lambda_0,\sigma_c^2)$ always satisfies $\sigma_c^2>\operatorname{Var}[Y]$.
    \item $g_Y(.;\sigma^2)$ may only become negative on $\mathbb{R}_+$ if and only if the previous set of equations has exactly one solution $(\lambda_0,\sigma_c^2)$  and $\sigma^2< \sigma_c^2$.
\end{itemize}
\end{proposition}

\Cref{Prop01Zeros} and \Cref{LemmaMerged} are interesting to obtain a characterization of the optimal variance proxy when $g_Y''(.;\sigma^2)$ has at most two zeros on each half-lines $\mathbb{R}_\pm^*$. Indeed, in this case the optimal proxy variance is obtained as the maximum of $\operatorname{Var}[Y]$, $\sigma_{c,+}^2$ and $\sigma_{c,-}^2$ where $\sigma_{c,+}^2$ (resp. $\sigma_{c,-}^2$) is the unique solution, when it exists, of the equations $g_Y(\lambda,\sigma^2)=0=g_Y'(\lambda,\sigma^2)$ on $\mathbb{R}_+^*$ (resp. $\mathbb{R}_-^*$). As we shall see, this situation happens for the $3$-mass distributions. It also includes many other standard distributions.

\section{Application to 3-mass distributions}\label{sec:3-mass}

In this section, we undertake a detailed analysis of the sub-Gaussian properties of three-mass discrete distributions supported on $\{-a,0,a\}$, see~\Cref{fig:distrib-plots}. Our objective is to characterize the optimal variance proxy $\sigma^2_{\mathrm{opt}}$ and to delineate the regimes in which strict sub-Gaussianity holds.  

In the \textit{symmetric case} (\Cref{subsec;symmthreepoints}), where the outer masses are equally weighted, we establish two distinct behaviors. When the parameter satisfies $p \geq \tfrac{1}{6}$, the distribution is strictly sub-Gaussian and the optimal variance proxy coincides with the variance, $\sigma^2_{\mathrm{opt}} = \mathrm{Var}[X] = 2p$. 
% \[
% \sigma^2_{\mathrm{opt}} = \mathrm{Var}[X] = 2p.
% \]
In contrast, for $p < \tfrac{1}{6}$, strict sub-Gaussianity fails, and the optimal variance proxy is determined by a critical parameter $\lambda_c > 0$ solving a coupled system of equations (\Cref{TheoSymmetric3mass}).  

In the \textit{asymmetric case}  (\Cref{subsec;asymmthreepoints}), where the mass probabilities at $-a$ and $a$ are not equal, the situation is more intricate. If the central mass satisfies $p_3 \leq 4\sqrt{p_1p_2}$, then an explicit closed-form expression is available (\Cref{TheoRegime1} and \Cref{Strictsub-GaussianityRegime1}), $\sigma^2_{\mathrm{opt}} = {2(p_2 - p_1)}/{\ln\!\left({p_2}/{p_1}\right)}$.
% \[
% \sigma^2_{\mathrm{opt}} = \frac{2(p_2 - p_1)}{\ln\!\left(\tfrac{p_2}{p_1}\right)}.
% \]
When $p_3 > 4\sqrt{p_1p_2}$, the characterization of $\sigma^2_{\mathrm{opt}}$ requires the analysis of a nonlinear equation whose solution yields the critical value determining the transition between variance proxies (\Cref{TheoRegime2}).  
 
%\textcolor{red}{Olivier: Can we put a summary of all important results of the asymmetric case. And a graph summarizing the situation.}
\subsection{Symmetric 3-mass distribution}\label{subsec;symmthreepoints}

Let $X$ be a discrete random variable on the set $\{-a,0,a\}$, $a>0$ and $    \P(X=-a)=p\,,\, \P(X=0)=1-2p\,,\, \P(X=a)=p$, 
% \begin{equation}
%     \P(X=-a)=p\,,\, \P(X=0)=1-2p\,,\, \P(X=a)=p
% \end{equation}
where $p\in \left( 0,\frac{1}{2}\right)$, see \Cref{fig:symmetric}.
We may define $Y=\frac{1}{a}X$ and use the fact that $\sigma_{\mathrm{opt}}[Y]=\frac{1}{a}\, \sigma_{\mathrm{opt}}[X]$. Thus, we may restrict to $a=1$. We have $\mu:=\E[Y]=0\,\,,\,\, \sigma^2:=\operatorname{Var}[Y]=2p$. In the special case where $p=\frac{1}{2}$ the random variable $X$ reduces to the symmetric Rademacher distribution, which is known to be strictly sub-Gaussian. 
For any $\sigma>0$, we have that $\sigma^2$ is a variance proxy of $Y$ if and only if
\beqq \E\big[e^{\lambda (Y-\mu)}\big]=p e^{\lambda }+ pe^{-\lambda } +1-2p\leq e^{\frac{\lambda^2\sigma^2}{2}} \,\, ,\,\, \forall \, \lambda\in \mathbb{R} .\eeqq
This inequality is equivalent to
\[
g_{\sigma,p}(\lambda):=\frac{\lambda^2\sigma^2}{2}-\ln( 2p\cosh \lambda +1-2p)\geq 0 \,\,,\,\,  \forall \, \lambda\in \mathbb{R} .
\]

Since $g_{\sigma,p}$ is an even function of $\lambda$, we only need to prove the former inequality for $\lambda\geq 0$.
%\beq g_{\sigma,p}(\lambda):=\frac{\lambda^2\sigma^2}{2}-\ln( 2p\cosh \lambda +1-2p)\geq 0 \,\,,\,\,  \forall \, \lambda\in \mathbb{R}_+ .\eeq
The general theory implies that $\operatorname{Var}[Y]=2p$ is a lower bound for the optimal variance proxy. Consequently, we shall only consider $\sigma^2\geq 2p$.
\medskip

The function $g_{\sigma,p}$ is obviously a smooth function of $(\lambda,\sigma,p)$ and we have
\beaa
    g'_{\sigma,p}(\lambda)&=&\lambda \sigma^2 -\frac{2p\sinh \lambda }{2p\cosh \lambda +1-2p}\,\,,\,\, 
g^{(2)}_{\sigma,p}(\lambda)=\sigma^2+\frac{2p\left(2p\cosh \lambda -\cosh \lambda -2p)\right)}{(2p\cosh \lambda +1-2p)^2}\cr
g^{(3)}_{\sigma,p}(\lambda)&=&\frac{2p\sinh \lambda\left(4p^2+4p-1+2p(1-2p)\cosh \lambda\right) }{(2p\cosh \lambda +1-2p)^3}.
\eeaa
Let us now observe that $\forall \, \lambda\in \mathbb{R}, \,N_{p}(\lambda):=4p^2+4p-1+2p(1-2p)\cosh \lambda \geq 1$,
and $g^{(3)}_{\sigma,p}(0)=0$. Thus, $g^{(3)}_{\sigma,p}$ and $N_{p}$ have the same sign on $\mathbb{R}_+$. Since $N_{p}'(\lambda)=2p(1-2p)\sinh \lambda$, we get that $N_{p}$ is a strictly increasing function on $\mathbb{R}_+$. Since $N_{p}(0)=6p-1$, we get two distinct cases.\medskip

\noindent\textbf{First}, $p\geq \frac{1}{6}$: In this case, $N_{p}$ is a positive function on $\mathbb{R}_+$ and so is $g^{(3)}_{\sigma,p}$. It follows that $g_{\sigma,p}^{(2)}$ is a strictly increasing function on $\mathbb{R}_+$. Since $g^{(2)}_{\sigma,p}(0)=\sigma^2-2p\geq0$, we conclude that $g^{(2)}_{\sigma,p}$ is positive  on $\mathbb{R}_+$. Thus $g'_{\sigma,p}$ is a strictly increasing function on $\mathbb{R}_+$. Furthermore, since $g_{\sigma,p}'(0)=0$ we end up with the fact that $g_{\sigma,p}'$ is a positive function on $\mathbb{R}_+$ and therefore $g_{\sigma,p}$ is an increasing function on $\mathbb{R}_+$. Finally, since $g_{\sigma,p}(0)=0$ we conclude that $g_{\sigma,p}$ is positive on $\mathbb{R}_+$ so that $\sigma$ is a variance proxy. This argument is valid for any $\sigma^2\geq \operatorname{Var}[Y]=2p$ so that $\sigma^2_{\mathrm{opt}}[Y]=\operatorname{Var}[Y]=2p$, i.e. {$Y$ is strictly sub-Gaussian}.\medskip

\noindent\textbf{Second}, $p<\frac{1}{6}$: In this case, $N_{p}(0)<0$, and $N_{p}$ is increasing on $\mathbb{R}_+$, tending to $+\infty$ when $\lambda\to +\infty$. Thus, since $N_{p}$ is a smooth function, there exists a unique $\lambda_0=\arcosh\left(\frac{1-4p-4p^2}{2p(1-2p)}\right)>0$ such that $N_{p}(\lambda_0)=0$. Moreover $N_{p}$ is strictly negative on $(0,\lambda_0)$ and strictly positive on $(\lambda_0,+\infty)$. These properties immediately extend to $g_{\sigma,p}^{(3)}$. Consequently, $g_{\sigma,p}^{(2)}$ is strictly decreasing on $(0,\lambda_0)$ and strictly increasing on $(\lambda_0,+\infty)$. In addition, we have $g_{\sigma,p}^{(2)}(0)=\sigma^2-2p\geq 0$ and $g_{\sigma,p}^{(2)}(\lambda_0)=\sigma^2-\frac{(1-2p)^2}{4(1-4p)}$ and $g_{\sigma,p}^{(2)}(+\infty)=\sigma^2>0$. Note that if $\sigma^2\geq \frac{(1-2p)^2}{4(1-4p)}$ then $g^{(2)}_{\sigma,p}$ is positive on $\mathbb{R}_+$ and then it is straightforward to prove that $\sigma$ is a variance proxy using the same final steps as the case $p\geq\frac{1}{6}$. Thus, we obtain the following upper bound for the optimal variance proxy of $Y$:
\begin{equation}\label{eq:sigma1}
        \sigma_{1}^2(p):=\frac{(1-2p)^2}{4(1-4p)}.
\end{equation}
Consequently $\sigma^2=2p$ is no longer a variance proxy because $g'_{\sigma,p}$ would become strictly decreasing on $(0,\lambda_0)$ and thus strictly negative on this interval (because $g'_{\sigma,p}(0)=0$) and so $g_{\sigma,p}$ would be negative on $(0,\lambda_0)$ (again because $g_{\sigma,p}(0)=0)$. This proves that {for $p<\frac{1}{6}$, $Y$ is not strictly sub-Gaussian}. \medskip
% \end{itemize}

Let us be more precise in the case $p<\frac{1}{6}$. As mentioned above, we shall now only consider $2p<\sigma^2<\frac{(1-2p)^2}{4(1-4p)}$. The previous analysis implies that there exists a unique pair $(\lambda_1,\lambda_2)$ such that $0<\lambda_1<\lambda_0<\lambda_2$ and $g_{\sigma,p}^{(2)}(\lambda_1)=g_{\sigma,p}^{(2)}(\lambda_2)=0$. Moreover, $g_{\sigma,p}^{(2)}$ is strictly positive on $(0,\lambda_1)\cup(\lambda_2,+\infty)$ and strictly negative on $(\lambda_1,\lambda_2)$. 
Note that we have explicitly:
\begin{align*}
    \lambda_1(\sigma)&:=\arcosh\left(  \frac{(1-2p)(1-2\sigma^2)-\sqrt{(1-2p)^2-4(1-4p)\sigma^2 } }{4p\sigma^2}\right)\\
\lambda_2(\sigma)&:=\arcosh\left(  \frac{(1-2p)(1-2\sigma^2)+\sqrt{(1-2p)^2-4(1-4p)\sigma^2 } }{4p\sigma^2}\right).
\end{align*}
This implies that $g_{\sigma,p}'$ is increasing on $(0,\lambda_1)$, then decreasing on $(\lambda_1,\lambda_2)$ and finally increasing on $(\lambda_2,+\infty)$. Since $g_{\sigma,p}'(0)=0$ and $g_{\sigma,p}'(+\infty)=+\infty$, the sign of $g_{\sigma,p}'$ is determined by the sign of $g_{\sigma,p}'(\lambda_2)$. Note also that a straightforward computation implies that $g_{\sigma,p}'(\lambda_0)=\sigma^2 \lambda_0\sqrt{(1-6p)(1+2p)(1-4p)}>0$. There are only two cases that we now study.\\

% \begin{itemize}
\noindent\textbf{Case when $g_{\sigma,p}'(\lambda_2)\geq 0$}: then $g_{\sigma,p}'$ is positive on $\mathbb{R}_+$ so that $g_{\sigma,p}$ is positive on $\mathbb{R}_+$ and thus $\sigma$ is a variance proxy. Thus an upper bound is $\sigma^2_{2}$ that is the unique solution in $\left(2p,\frac{(1-2p)^2}{4(1-4p)}\right)$ of 
\beq \label{eq:sigma2}
\sigma^2_{2}(p)=\frac{2p\sinh\lambda_2(\sigma_{2}(p))}{\lambda_2(\sigma_{2}(p)) (1+2p\cosh \lambda_2(\sigma_{2}(p))-2p) }.\eeq
% where $\lambda_2(x)=\arcosh\left(  \frac{(1-2p)(1-2x^2)+\sqrt{(1-2p)^2-4(1-4p)x^2 } }{4px^2}\right)$.\\

\noindent\textbf{Case when $g_{\sigma,p}'(\lambda_2)<0$}: then there exists a unique pair $(\lambda_3,\lambda_4)$ such that $\lambda_0<\lambda_3<\lambda_2<\lambda_4$ and $g_{\sigma,p}'(\lambda_3)=g_{\sigma,p}'(\lambda_4)=0$. Moreover, $g_{\sigma,p}'$ is positive on $(0,\lambda_3)\cup(\lambda_4,+\infty)$ and negative on $(\lambda_3,\lambda_4)$. This implies that $g_{\sigma,p}$ increases on $(0,\lambda_3)$, then decreases on $(\lambda_3,\lambda_4)$ and finally increases on $(\lambda_4,+\infty)$. In particular, since $g_{\sigma,p}(0)=0$ and $g_{\sigma,p}(+\infty)=+\infty$, $g_{\sigma,p}$ has only one local maximum $\lambda_3$ and one local minimum $\lambda_4$ on $(0,+\infty)$. Moreover, these local extremum satisfy $0<\lambda_3<\lambda_0<\lambda_4$ and $g_{\sigma,p}(\lambda_3)>0$. Eventually the sign of $g_{\sigma,p}(\lambda_4)$ determines if $\sigma$ is a variance proxy or not. Since $g_{\sigma,p}(\lambda_4)$ is a smooth function of $\sigma$, the critical case corresponding to $\sigma_{\mathrm{opt}}$ may happen only when $g_{\sigma_{\mathrm{opt}},p}(\lambda_4)=0$. Since we know that this critical case is achieved on $\left(2p,\frac{1-4p+4p^2}{4(1-4p)}\right)$ we conclude with the following statement.
% \end{itemize}

%\begin{comment}
\begin{figure}[tb]% \vspace{-1\baselineskip}
\centering
\scalebox{.78}{
\input{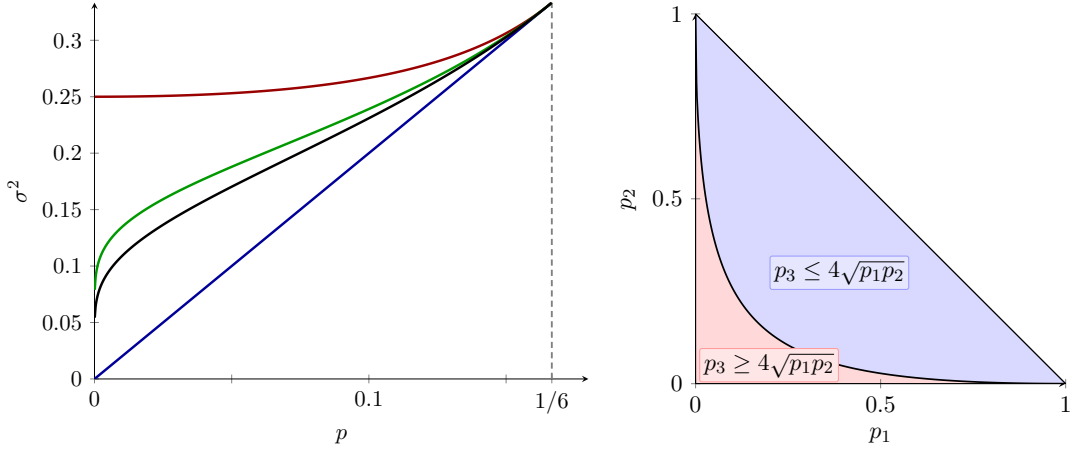}
}
\scalebox{.82}{
\begin{tikzpicture}
\begin{axis}[
  width=6cm, height=6cm,
  scale only axis,
  axis equal image,           % same unit scale on x and y
  axis lines=left,
  xmin=0, xmax=1,
  ymin=0, ymax=1,
  xlabel={$p_1$},
  ylabel={$p_2$},
  xtick={0,0.5,1},
  ytick={0,0.5,1},
  ticklabel style={/pgf/number format/fixed},
  legend style={
    at={(rel axis cs:0.96,0.96)}, anchor=north east,
    fill=white, draw=none, font=\small
  }
]

% --- Define the triangle boundary as a named path (feasible region) ---
\path[name path=tri] (axis cs:0,0) -- (axis cs:1,0) -- (axis cs:0,1) -- cycle;

% --- Curve: (1 - x - y)^2 = 16 x y -> y(x) = 1 + 7x - 4*sqrt(x*(1+3x)) ---
\addplot[name path=curve, domain=0:1, samples=400, smooth, thick]
  {1 + 7*x - 4*sqrt(x*(1 + 3*x))};

% --- Diagonal boundary: p1 + p2 = 1 ---
\addplot[name path=diag, domain=0:1, samples=2, thin] {1 - x};

% % Feasible triangle (x>=0, y>=0, x+y<=1), lightly red by default:
\addplot[draw=none, fill=red!15, on layer=axis background] coordinates {(0,0) (1,0) (0,1)} -- cycle;

% --- BLUE above the curve: between diagonal and curve (this overlays and thus "carves out" the red) ---
\addplot[blue!15, on layer=axis foreground] fill between[of=diag and curve];

% --- Draw triangle edges for clarity ---
\addplot[thin] coordinates {(0,0) (1,0)};
\addplot[thin] coordinates {(0,0) (0,1)};
\addplot[thin] coordinates {(0,1) (1,0)};

% --- inline labels inside the filled regions ---
\node[anchor=west, fill=blue!10, draw=blue!40, rounded corners=1pt, inner sep=2pt]
  at (axis cs:0.20,0.3) {$p_3 \le 4\sqrt{p_1 p_2}$};

\node[anchor=west, fill=red!10, draw=red!40, rounded corners=1pt, inner sep=2pt]
  at (axis cs:0.01,0.05) {$p_3 \ge 4\sqrt{p_1 p_2}$};
\end{axis}
\end{tikzpicture}
}
\caption{Left: Symmetric 3-mass case (Section \ref{subsec;symmthreepoints}). Black: Optimal variance proxy $\sigma_{\mathrm{opt}}^{2}$ for $0<p<\frac{1}{6}$. Blue: variance $\operatorname{Var}[Y] = 2p$. Red: upper bound $\sigma_{1}^2(p)$ of \Cref{eq:sigma1}. Green: upper bound $\sigma_{2}^2(p)$ of \Cref{eq:sigma2}. Right: Asymmetric 3-mass case (Section \ref{subsec;asymmthreepoints}). Two different regimes depending on the relative weight of the intermediate mass.}
\label{fig:3-mass-illustrations}
\end{figure}
%\end{comment}

\begin{theorem}[Optimal variance proxy for symmetric 3-mass distribution.]\label{TheoSymmetric3mass}
Let $X$ be a discrete random variable on the set $\{-a,0,a\}$ where $a > 0$ \text{ and }  $\P(X=-a)=p\,,\, \P(X=0)=1-2p\,,\, \P(X=a)=p$  where $p\in \left( 0,\frac{1}{2}\right)$. Then we have two regimes:
\begin{enumerate}
\item[\textbf{(i)}] \emph{Strictly sub-Gaussian regime.}
If \(p \in \left[\tfrac16, \tfrac12\right)\), then \(X\) is strictly sub-Gaussian, i.e.,
\beqq \sigma_{\mathrm{opt}}^{2} = \operatorname{Var}[X] = 2p. \eeqq
\item[\textbf{(ii)}] \emph{Non-strictly sub-Gaussian regime.}  
If \(p \in (0,\tfrac16)\), then \(X\) is not strictly sub-Gaussian. In this case, the optimal variance proxy is characterized by the system
\beqq
\begin{cases}
0 = g_{\sigma_{\mathrm{opt}},p}(\lambda_c)
   = \tfrac{\lambda_c^2\sigma_{\mathrm{opt}}^2}{2}
     - \ln\!\big(2p\cosh \lambda_c + 1 - 2p\big),\\[1ex]
0 = g'_{\sigma_{\mathrm{opt}},p}(\lambda_c)
   = \lambda_c \sigma_{\mathrm{opt}}^2
     - \tfrac{2p\sinh \lambda_c}{2p\cosh \lambda_c + 1 - 2p},
\end{cases}
\eeqq
where $\lambda_c \in (\lambda_0, +\infty)$ is the unique solution with \beqq \lambda_0 = \arcosh\!\Bigg(\frac{1 - 4p - 4p^2}{2p(1 - 2p)}\Bigg) > 0.\eeqq
\end{enumerate}
\end{theorem}
Equivalently, the optimal variance proxy of the non-strictly sub-Gaussian regime admits the closed-form \beqq
\sigma_{\mathrm{opt}}^{2} = \frac{2p\,\sinh(\lambda_{c})} {\lambda_{c}\,\big(2p\cosh(\lambda_{c}) + 1 - 2p\big)}, \eeqq
where $\lambda_c$ is the unique solution of \beqq p\,\lambda_{c}\,\sinh(\lambda_{c}) - \big(1 - 2p + 2p\cosh(\lambda_{c})\big) \ln\!\big(1 - 2p + 2p\cosh(\lambda_{c})\big) = 0. \eeqq
In this regime, $\operatorname{Var}[X]=2p < \sigma_{\mathrm{opt}}^{2} < \sigma_1^2=\tfrac{(1-2p)^2}{4(1-4p)}$.

\subsection{Asymmetric 3-mass distribution}\label{subsec;asymmthreepoints}

Let $X$ be a random variable supported on $\{-a,0,a\}$, $a>0$, and $\P(X=-a)=p_1\,,\, \P(X=0)=p_3 = 1-p_1-p_2\,,\, \P(X=a)=p_2$, 
where $p_1, p_2, p_3\in \left( 0,1\right)$, see \Cref{fig:asymmetric}. As before, we may define $Y=\frac{1}{a}X$ and restrict to $a=1$. We also may assume $p_2 \geq p_1$ without loss of generality.  We have
$\mu:=\E[Y]=-p_1+p_2 \geq 0$ and $\operatorname{Var}[Y]=p_1+p_2-(-p_1+p_2)^2$. 
We have that $\sigma^2$ is a variance proxy of $Y$ if and only if
\beqq \E\big[e^{\lambda Y}\big]=p_1 e^{-\lambda }+ p_2e^{\lambda } +1-p_1-p_2\leq e^{(\frac{\lambda^2\sigma^2}{2}+\lambda \mu)} \,\, ,\,\, \forall \, \lambda\in \mathbb{R} .\eeqq
This inequality is equivalent to
\beq \label{eq:gp12}
g_{\sigma,p_1,p_2}(\lambda):=\frac{\lambda^2\sigma^2}{2}-\ln(u_{0;p_1, p_2}(\lambda))+\lambda \mu \geq 0 \,\,,\,\,  \forall \, \lambda\in \mathbb{R},\eeq
with  $u_{0;p_1, p_2}(\lambda) := p_1e^{-\lambda}+p_2e^{\lambda}+1-p_1-p_2$ a smooth function and  $u_{0;p_1, p_2}(\lambda) > 0 \text{ , } \forall \lambda \in \mathbb{R}$. For convenience, we drop the dependence in $(p_1, p_2)$ in the notation, as per $u_0(\lambda)$, and we define $u_{i}(\lambda)=u_{i-1}'(\lambda),   \  i  \in \{1,2,3\}$, $\forall \, \lambda\in \mathbb{R}$. 
We have $u_1(\lambda) = - p_1e^{-\lambda} + p_2e^{\lambda}$ and we observe that $u_2(\lambda) = u_0(\lambda)-p_3$, $u_3(\lambda) = u_1(\lambda)$ and $u_1^2(\lambda)=(u_0(\lambda)-p_3)^2-4p_1p_2$ for all $\lambda \in \mathbb{R}$. 
The general theory implies that $\operatorname{Var}[Y]=p_1+p_2-(-p_1+p_2)^2$ is a lower bound for the optimal variance proxy, thus we shall only consider larger or equal values for $\sigma^2$.
%\medskip

The function $g_{\sigma,p_1,p_2}$ is a smooth function of $(\lambda,\sigma,p_1,p_2)$ and its first derivatives can be expressed as:
\bea
    g'_{\sigma,p_1,p_2}(\lambda)&=&\lambda \sigma^2 -\frac{u_1(\lambda)}{u_0(\lambda)} + \mu \,\,,\,\, 
g^{(2)}_{\sigma,p_1,p_2}(\lambda)= \sigma^2-\frac{u_0(\lambda)p_3-p_3^2+4p_1p_2}{u_0(\lambda)^2}\cr
g^{(3)}_{\sigma,p_1,p_2}(\lambda)&=& u_1(\lambda)\frac{p_3u_0(\lambda)+8p_1p_2-2p_3^2}{u_0(\lambda)^3}.
\eea
Note in particular that $g^{(3)}_{\sigma,p_1,p_2}$ does not depend on $\sigma$. Moreover, $g^{(3)}_{\sigma,p_1,p_2}$ can be written as
\beq
\label{g3} 
g^{(3)}_{\sigma,p_1,p_2}(\lambda) = \frac{u_1(\lambda)N_{p_1,p_2}(\lambda)}{u_0(\lambda)^3}, \,\,\text{where}\,\, 
N_{p_1,p_2}(\lambda) := p_3u_0(\lambda)+8p_1p_2-2p_3^2.
\eeq
Observe that $u_0(\lambda)$ is strictly positive on $\mathbb{R}$. Therefore, the function $g_{\sigma,p_1,p_2}^{(3)}$ and $u_1 N_{p_1, p_2} $ share the same sign. We know that $u_1$ changes sign at $\lambda_0 := -\frac{1}{2}\ln(\frac{p_2}{p_1})$ assuming that $p_2\geq p_1$ we have $\lambda_0 \leq 0$. Hence, to determine the sign of $g^{(3)}_{\sigma,p_1,p_2}$, it remains to analyze the sign of $N_{p_1, p_2}$. For this purpose, it is sufficient to  study the sign of the polynomial
\begin{equation}\label{eq:P-def}
    P(X) :=p_2p_3X^2+(8p_1p_2-p_3^2)X+p_1p_3, \text{where}\  X = e^{\lambda} > 0.
\end{equation}
To determine the sign of $P(X)$, we compute its discriminant:
\begin{equation*}\label{eq:delta-def}
    \Delta = (p_3^2)^2 - 20p_1p_2(p_3)^2 + 64p_1^2p_2^2= (p_3^2-4p_1p_2)(p_3^2-16p_1p_2),
\end{equation*}
and it gives two different regimes, separated as illustrated on the right panel of Figure~\ref{fig:3-mass-illustrations}, that we shall study separately. 
%
% \begin{figure}[!ht]
% \begin{center}
% \begin{tikzpicture}
% \input{figures/illustration_asym_limits}
% \end{tikzpicture}
% \end{center}
% \caption{Two different regimes in the asymmetric 3-mass case.}
% \label{fig:illustration_asym_limits}
% \end{figure}
% \begin{wrapfigure}{R}{0.5\textwidth}
% \begin{center}
% \begin{tikzpicture}
% \input{figures/illustration_asym_limits}
% \end{tikzpicture}
% \end{center}
% \caption{Asymmetric 3-mass case. Two different regimes.}
% \label{fig:illustration_asym_limits}
% \end{wrapfigure}

\subsubsection{Case $p_3\leq 4\sqrt{p_1p_2}$}
In this case, we have the following theorem. %See the right panel of \Cref{fig:3-mass-illustrations} for an illustration. 
\begin{theorem}[Optimal variance proxy for $p_3\leq 4\sqrt{p_1p_2}$, closed-form expression.]\label{TheoRegime1} 
When $p_3\leq 4\sqrt{p_1p_2}$, the optimal variance proxy is given by
\beqq \sigma_{\mathrm{opt}}^2= \frac{2(p_2-p_1)}{\ln p_2 - \ln p_1}.\eeqq
\end{theorem}
The proof is deferred to the Appendix.
\Cref{TheoRegime1} implies the following corollary.
\begin{corollary}\label{Strictsub-GaussianityRegime1} For $p_3\leq 4\sqrt{p_1p_2}$, the random variable $Y$ is strictly sub-Gaussian if and only if $p_1=p_2=p$. In this symmetric case, the condition $p_3\leq 4\sqrt{p_1p_2}$ is equivalent to $p\geq \frac{1}{6}$. 
\end{corollary}
\begin{proof}It is obvious from the fact that in this case $\sigma_{\mathrm{opt}}^2=\frac{2(p_2-p_1)}{\ln(p_2/p_1)}\geq \operatorname{Var}[Y]=p_1+p_2-(p_2-p_1)^2$ with equality if and only if $p_1=p_2$.
\end{proof}
% \bigskip
%\begin{comment}
\begin{figure}[htbp]
  \centering
  \input{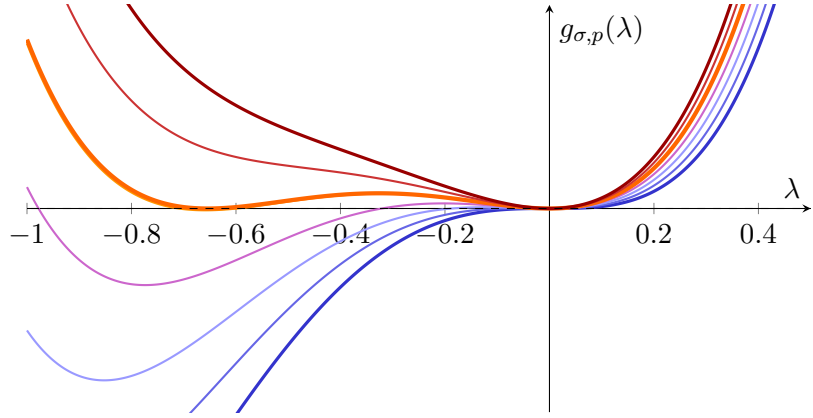}
  \caption{Functions $\lambda \mapsto g_{\sigma,p_1,p_2}(\lambda)$ of Equation~\eqref{eq:gp12} for $(p_1, p_2) = (0.13, 0.25)$, with $\sigma^2$ varying from the upper bound $\frac{2\sqrt{p_1p_2}}{p_3+2\sqrt{p_1p_2}}$ in \textcolor{darkred}{dark red} down to the variance $\text{Var}[Y]$ in \textcolor{blue}{blue}  (the function $g$ then becomes locally negative around $\lambda = 0_-$), while the \textcolor{orange}{orange} curve stands for the optimal proxy variance $\sigma_{\mathrm{opt}}^2= \frac{2(p_2-p_1)}{\ln p_2 - \ln p_1}$. The intermediate curves illustrate the progressive transition from convex behavior around 0 to oscillating behavior.}
  \label{fig:/figures/g_plot_p1p2}
\end{figure}
%\end{comment}

\subsubsection{Case $p_3>4\sqrt{p_1p_2}$}
In this case, the discriminant $\Delta$ defined in Equation~\eqref{eq:delta-def} is positive and  the polynomial $P$ of Equation~\eqref{eq:P-def} has two distinct positive roots so that $N_{p_1,p_2}$ defined in Equation~\eqref{g3} has exactly two roots. The following lemma provides the location of the latter roots.

\begin{lemma}
\label{lemma:lambda_ordering}
Let $p_1, p_2 \in (0,1)$ such that $p_2 \geq p_1$ and $(1-p_1-p_2)^2>16p_1p_2$. Then $N_{p_1,p_2}$ defined in Equation~\eqref{g3} has exactly two roots which are \emph{positive}.    
\end{lemma}

We now analyze the sign of the function  $g_{\sigma,p_1,p_2}(.;\sigma^2)$  separately  on $\mathbb{R}_-^*$ and $\mathbb{R}_+^*$. Let us first observe that the discussion made for $\mathbb{R}_-^*$ in the proof of \Cref{TheoRegime1} is still valid. Thus, $g_{\sigma,p_1,p_2}(.;\sigma^2)$ is non-negative on $\mathbb{R}_-^*$ if and only if $\sigma^2\geq \frac{2(p_2-p_1)}{\ln(p_2/p_1)}$ (see Figure~\ref{fig:/figures/g_plot_p1p2}). Let us now discuss the situation on $\mathbb{R}_+^*$. From \Cref{lemma:lambda_ordering} and \eqref{g3} we get that $g^{(3)}$ is negative on $(-\infty,\lambda_0)\cup\left(\lambda_-(\sigma),\lambda_+(\sigma)\right)$  and positive on $(\lambda_0,\lambda_-(\sigma))\cup\left(\lambda_+(\sigma),+\infty\right)$. Consequently, $g^{(2)}_{\sigma,p_1,p_2}(.;\sigma^2)$ is increasing on $(0,\lambda_-)$ and since $g^{(2)}_{\sigma,p_1,p_2}(0;\sigma^2)=0$ is it thus positive on this interval. Then, $g_{\sigma,p_1,p_2}^{(2)}(.;\sigma^2)$ is decreasing on $(\lambda_-(\sigma),\lambda_+(\sigma))$ and then increasing on $(\lambda_+(\sigma),+\infty)$ with $g_{\sigma,p_1,p_2}^{(2)}(+\infty;\sigma^2)=\sigma^2>0$. Thus, there are only two possible cases: either $g_{\sigma,p_1,p_2}^{(2)}(\lambda_+(\sigma);\sigma^2)\geq 0$ and then $g_{\sigma,p_1,p_2}(.;\sigma^2)$ is strictly convex and positive on $\mathbb{R}_+^*$ or $g^{(2)}_{\sigma,p_1,p_2}(\lambda_+(\sigma);\sigma^2)<0$ in which case $g_{\sigma,p_1,p_2}^{(2)}(.;\sigma^2)$ has exactly two zeros $(\lambda_1(\sigma),\lambda_2(\sigma))$ on $\mathbb{R}_+^*$ such that $\lambda_-(\sigma)<\lambda_1(\sigma)<\lambda_+(\sigma)<\lambda_2(\sigma)$. We may thus apply \Cref{LemmaMerged} whose conclusion depends on the existence of a solution $(\lambda,\sigma^2)\in \left(\lambda_-,+\infty\right)\times \left(\operatorname{Var}[Y],+\infty\right)$ of the equation $g_{\sigma,p_1,p_2}(\lambda;\sigma^2)=0=g_{\sigma,p_1,p_2}'(\lambda;\sigma^2)$. This set of equations is equivalent to
\[
\sigma^2=\frac{1}{\lambda}\left(\frac{u_1(\lambda)}{u_0(\lambda)}-\mu\right)=\frac{1}{\lambda}\left(p_1-p_2+\frac{p_2e^\lambda-p_1e^{-\lambda}}{p_1e^{-\lambda}+p_2e^{\lambda}+1-p_1-p_2}\right),
\] 

and $\lambda$ positive solution of
\[
\lambda \frac{u_1(\lambda)}{u_0(\lambda)}- 2\ln (u_0(\lambda))+\lambda \mu=0 ,
\] 
i.e.
\beq
 \label{EqF} F(\lambda):=\lambda u_1(\lambda)-2u_0(\lambda)\ln u_0(\lambda)+ \lambda u_0(\lambda)(p_2-p_1)=0. \eeq
From \Cref{LemmaMerged}, $F(\lambda)=0$ has no solution on $\mathbb{R}_+^*$ such that $\sigma^2=\frac{1}{\lambda}\left(\frac{u_1(\lambda)}{u_0(\lambda)}-\mu\right)<\operatorname{Var}[Y]$. Thus, we have two cases detailed in the following theorem.
\begin{theorem}[Optimal variance proxy for $p_3> 4\sqrt{p_1p_2}$.]\label{TheoRegime2} 
When $p_3> 4\sqrt{p_1p_2}$, the optimal variance proxy depends on the zero of $F$ in \Cref{EqF}:
\begin{itemize}
    \item If $F$ has a positive zero $\lambda_c>0$ then \Cref{LemmaMerged} implies that this zero is unique on $\mathbb{R}_+^*$ and we get that $\sigma_c^2=\frac{1}{\lambda_c}\left(\frac{u_1(\lambda_c)}{u_0(\lambda_c)}-\mu\right)\geq\operatorname{Var}[Y]$. \Cref{GeneralCharac} implies that the optimal variance proxy is given by
    \[
     \sigma_{\mathrm{opt}}^2=\max\left(\frac{2(p_2-p_1)}{\ln(p_2/p_1)}, \frac{1}{\lambda_c}\left(\frac{u_1(\lambda_c)}{u_0(\lambda_c)}-(p_2-p_1)\right) \right).
     \]
    \item If $F$ has no zero on $\mathbb{R}_+^*$ then $g_{\sigma,p_1,p_2}(.;\sigma^2)$ is positive on $\mathbb{R}_+^*$ and thus the optimal variance proxy is given by $\sigma^2_{\text{opt}}=\frac{2(p_2-p_1)}{\ln(p_2/p_1)}$.
\end{itemize}
\end{theorem}
\begin{remark}Note that $\underset{\lambda\to +\infty}{\lim} F(\lambda)=-\infty$ and $F(\lambda)=\frac{1}{6}M_3[Y]\lambda^3+\frac{1}{4}\left(\frac{1}{3}M_4[Y]-\operatorname{Var}[Y]^2\right)\lambda^4+ O(\lambda^5)$, where $M_3[Y]:=\mathbb{E}[(Y-\mu)^3]$ and $M_4[Y]:=\mathbb{E}[(Y-\mu)^4]$. Thus, a sufficient condition for $F$ to admit a positive zero is that $M_3[Y]>0$ or $\left[M_3[Y]=0 \text{ and } \frac{1}{3}M_4[Y]\operatorname{Var}[Y]^2>0\right]$. In particular for $p_1=p_2=p$, we have $F(\lambda)=\frac{1}{6}p(1-6p)\lambda^4+O(\lambda^5)$ so that for $p>\frac{1}{6}$ (which is the equivalent condition to $p_3>4\sqrt{p_1p_2}$ in this case) $F$ always has a unique positive zero $\lambda_c$ and the corresponding variance $\sigma_c^2=\frac{1}{\lambda_c}\frac{u_1(\lambda_c)}{u_0(\lambda_c)}$ is always greater than $2p=\operatorname{Var}[Y]$ and hence the optimal variance proxy is always $\sigma_{\mathrm{opt}}^2=\sigma_c^2=\frac{1}{\lambda_c}\frac{u_1(\lambda_c)}{u_0(\lambda_c)}$ in the symmetric case $p_1=p_2=p$.
\end{remark}

\section{Optimal variance proxy for the discrete uniform distribution}\label{sec:uniform}

In this section, we briefly present the discrete, equally spaced uniform distribution, also known as the comb distribution, see \Cref{fig:uniform}. 
Our main result establishes that this law is strictly sub-Gaussian, 
with an optimal variance proxy that coincides with its variance. 

\begin{theorem}[Optimal variance proxy for the uniform discrete distribution.]
\label{TheoRDisceteDistribution}
Let $X$ be uniformly distributed on $\{ka+b\,,\,  k\in \llbracket1,N\rrbracket\}$ with $N>1$, 
$a\in\mathbb{R}\setminus\{0\}$ and $b\in \mathbb{R}$. 
Then $X$ is strictly sub-Gaussian, i.e. its optimal variance proxy equals its variance: 
\[
 \sigma^2_{\mathrm{opt}}[X]=\operatorname{Var}[X]=a^2\frac{N^2-1}{12}.
\]
\end{theorem}

The proof of \Cref{TheoRDisceteDistribution} is postponed to \appendixref{appendix:Appendix_proof_th13}.

\section{Software implementation}\label{sec:software}

To support reproducibility and facilitate the use of both state-of-the-art and our theoretical results, we developed a Python package \footref{note1} 
%$^\ref{note1}$
%\footnotemark[\ref{note1}]
% \footnote{
% \sloppy{The package is  available at \url{https://github.com/jarbel/sub-Gaussian-implementation.git} with comprehensive documentation, installation instructions, and usage examples.}
% } 
to compute the optimal sub-Gaussian variance proxy for a wide range of probability distributions. The package handles both discrete and continuous distributions, including truncated Gaussian and Exponential laws.

The implementation follows a unified principle. Whenever a closed-form expression is available, it is returned directly; this is the case for Bernoulli and Binomial distributions, uniform and discrete uniform distributions (see \Cref{TheoRDisceteDistribution}), as well as certain symmetric families such as Beta, Kumaraswamy, Triangular \citep[see][]{arbel2020strict}, or 3-mass distributions (see \Cref{TheoSymmetric3mass}). When no closed-form can be derived, the computation relies on the general characterizations of the optimal variance proxy given in \sectionref{sec:charac} (see also \citep{arbel2020strict}), and proceeds via numerical methods. In particular, an adaptive grid search is combined with robust root-finding algorithms such as Brent’s method. This hybrid methodology ensures that both analytically tractable and intractable cases are encompassed within a single coherent framework.

% \subsection{Alternative formulations and validation}
The package primarily implements the characterization of the optimal variance proxy based on function $g_Y$ defined in Equation~\eqref{eq:g_Y} and based on the cumulant-generating function of \(Y-\mu\). Computing \(\sigma^2_{\mathrm{opt}}\) then reduces to solving the coupled system  
\[
g_Y(\lambda;\sigma^2)=0, \qquad g'_Y(\lambda;\sigma^2)=0,
\]  
as stated in \Cref{TheoRegime2}. This characterization is particularly effective for discrete distributions with few support points, such as the 3-mass distribution, where it provides a tractable criterion for identifying candidate values of the variance proxy.  

We also used this implementation to validate our theoretical results: the specialized method for the symmetric 3-mass case (see \Cref{TheoSymmetric3mass}) was checked against the asymmetric 3-mass case (see \Cref{TheoRegime2}) for the special scenario where \(p_1 = p_2\), and both approaches produced identical results. This provides an additional consistency check between the theoretical framework and the numerical implementation.  

The package provides utilities for visualization  of objective functions and supports batch analysis across multiple parameter settings, making it a practical companion for applied research in Bayesian inference, variational methods, and concentration bounds.

\section{Discussion}

In this work, we advance the understanding of the sub-Gaussian property for discrete distributions by deriving the optimal sub-Gaussian variance proxy for certain 3-mass distributions, in particular those with equally spaced support such as $\{-1,0,1\}$. We further extend the analysis to the uniform discrete distribution on $\{1,\ldots,N\}$ with equally spaced support.

Generalizing beyond these settings to distributions with non-equidistant support or with more than 3 mass points, appears essentially intractable in full generality for $N$-mass categorical laws. Nonetheless, other discrete families remain of substantial interest. In Bayesian nonparametrics, for instance, one often encounters discrete distributions, e.g. those arising from the Dirichlet process \citep{ferguson1973bayesian}, the Pitman--Yor process \citep{pitman1997two}, or Gibbs-type processes \citep{deblasi2015gibbs}. While concentration properties for such processes have been studied in the context of large deviations \citep[e.g.,][]{doss1982tails,feng2007large}, a more refined analysis of their tails via optimal variance proxies represents a promising direction for future research.

%In addition, we provide a general characterization of the optimal sub-Gaussian variance proxy for random variables with bounded moment-generating functions, using a methodology based on function variation analysis. To facilitate further research and applications, we also release an open-source Python package that combines analytical and numerical techniques for computing optimal sub-Gaussian variance proxies across a wide range of distributions. We hope this resource will serve as a foundation for both theoretical advances, by enabling exploration of new families of distributions, and practical applications, where precise sub-Gaussian parameters can improve the design and analysis of probabilistic algorithms.

% Acknowledgments---Will not appear in anonymized version
\section*{Acknowledgment}
Olivier Marchal used part of his IUF junior grant G752IUFMAR for this research, and Julyan Arbel was partially supported by ANR-21-JSTM-0001 grant.

\newpage

\appendix

% \\efalias{section}{appendix} % uncomment if you are using cleveref

\section{Proofs for \sectionref{sec:charac}} \label{appendix:tech-lemmas}

In this section we prove \Cref{AsymptInf}, \Cref{GeneralCharac}, and \Cref{Prop01Zeros}. We also establish two lemmas,  \Cref{LemmaNumber} and \Cref{LemmaNumber2},  which are useful to prove \Cref{LemmaMerged}.

\begin{proofof}{\Cref{AsymptInf}} It is obvious from the definition of $g_Y(\lambda;\sigma^2)=\frac{1}{2}\lambda^2\sigma^2-M_Y(\lambda)$. The sufficient condition is also immediate since if $M$ is a bound for $Y$ then
$|M_Y(\lambda)|\leq \lambda|M+\mu|=o(\lambda^2)$. The fact that $Y$ is sub-Gaussian follows from that there exists $C>0$ such that $|M_Y(\lambda)|\leq C\lambda^2$ for all $\lambda\in \mathbb{R}$. Thus, taking $\frac{C}{2}$ immediately gives that $\sigma^2$ is a variance proxy so that $Y$ is sub-Gaussian.
\end{proofof}

% First we prove \Cref{GeneralCharac}.

\begin{proofof}{\Cref{GeneralCharac}}
Let us first mention that \Cref{AsymptInf} implies that the optimal variance proxy is well-defined. Then let us consider $\sigma^2>\max\{ \operatorname{Var}[Y],\; \sup\mathcal{S}_c^* \}$. Since $\sigma^2>\operatorname{Var}[Y]$, we get that $g_Y(.;\sigma^2)$ is locally convex and non-negative around $\lambda=0$.  Let us consider $\lambda_m(\sigma)\neq 0$ a local minimum of $g_Y(.,\sigma^2)$. For simplicity we shall assume that $\lambda_m(\sigma)>0$ but a similar argument is valid if $\lambda_m(\sigma)<0$. Then we have $\partial_{\sigma}[g_Y(\lambda_m(\sigma);\sigma^2)]= g_Y'(\lambda_m(\sigma);\sigma^2) \partial_\sigma \lambda_m(\sigma) + \lambda_m(\sigma)^2\sigma= \lambda_m(\sigma)^2\sigma>0$. Assume by contradiction that $g_Y(\lambda_m(\sigma);\sigma^2)<0$ then increasing $\sigma$ would increase the value of $g_Y(\lambda_m(\sigma);\sigma^2)$. Since $\partial_{\sigma}[g_Y(\lambda_m(\sigma);\sigma^2)]=\lambda_m(\sigma)^2\sigma>0$ there are only two cases: \\

\noindent\textbf{First}, $\lambda_m(\sigma)$ remains outside a positive neighborhood of $0$ denoted $(0,\epsilon)$ when we increase $\sigma$ and thus since $\partial_{\sigma}[g_Y(\lambda_m(\sigma);\sigma^2)]>\epsilon^2 \sigma$ and by assumption $g_Y(\lambda_m(\sigma);\sigma^2)<0$, there exists a value $\sigma_1>\sigma$ for which $g_Y(\lambda_m(\sigma_1);\sigma_1^2)=0$ which is a contradiction because $\sigma_1^2\in \mathcal{S}_c^*$ so that we should have $\sigma\geq \sigma_1$.\\

\noindent\textbf{Second}, $\lambda_m(\sigma)\to 0_+$ when we increase $\sigma$. In this case this is a contradiction because $g_Y(.,\sigma^2)$ is locally positive and convex in a positive neighborhood of $0$. Indeed, $g_Y(.,\sigma^2)$ must reach a positive local maximum on $(0,\lambda_m(\sigma))$ that we denote $\lambda_{\text{max}}(\sigma)$. By Rolle's theorem on $(0,\lambda_{\text{max}}(\sigma))$ and $(\lambda_{\text{max}}(\sigma),\lambda_m(\sigma))$, there exist at least two distinct values $(\lambda_1(\sigma),\lambda_2(\sigma))$ with $0<\lambda_1(\sigma)<\lambda_{\text{max}}(\sigma)<\lambda_2(\sigma)<\lambda_m(\sigma)$ such that $g_Y''(\lambda_1(\sigma);\sigma^2)=g_Y''(\lambda_2(\sigma);\sigma^2)=0$. Since $\lambda_m(\sigma)\to 0_+$, we must also have $\lambda_1(\sigma),\lambda_2(\sigma)\to 0_+$ and hence by continuity $g_Y''(0,\sigma^2)\to 0$. But this is impossible since $g_Y''(0;\sigma^2)=(\sigma^2-\operatorname{Var}[Y])>0$ is a positive, increasing function of $\sigma$. \\

Thus we conclude that for any $\sigma^2>\max\{\operatorname{Var}[Y],\sup\mathcal{S}_c^*\}$, all local minima of $g_Y(.;\sigma^2)$ are non-negative so that since $\underset{\lambda\to \pm \infty}{\lim} g_Y(\lambda;\sigma^2)=+\infty$, $g_Y(.;\sigma^2)$ is non-negative on $\mathbb{R}$. Hence $\sigma^2$ is a variance proxy and thus $\sigma_{\mathrm{opt}}^2\leq \max\{\operatorname{Var}[Y],\sup\mathcal{S}_c^*\}$.

\medskip

\sloppy{Let us prove the converse inequality and assume that $\sigma_{\mathrm{opt}}^2< \max\{\operatorname{Var}[Y],\sup\mathcal{S}_c^*\}$. It is well-known that $\operatorname{Var}[Y]$ is always a lower bound for the optimal variance proxy, thus the last inequality is only possible if  $\operatorname{Var}[Y]<\sigma_{\mathrm{opt}}^2< \sup\mathcal{S}_c^*$.} Thus, there exists $s_c\in \mathcal{S}_c^*$ such that $\sigma_{\mathrm{opt}}^2<s_c$, i.e. there exists $\lambda_c\in \mathbb{R}^*$ such that $g_Y(\lambda_c;s_c)=g_Y'(\lambda_c;s_c)=0$ and $\lambda_c$ is a local minimum of $g_Y(.;s_c)$. We have from the fact that the dependence of $g_Y$ relatively to $\sigma$ is quadratic that:
\beq g_Y(\lambda_c,\sigma^2)=g_Y(\lambda_c,s_c) +\frac{1}{2}\lambda_c^2 (\sigma^2-s_c)=\frac{1}{2}\lambda_c^2 (\sigma^2-s_c).\eeq
so that $g_Y(\lambda_c,\sigma^2)<0$ when $\sigma^2<s_c$. This implies that $g_Y(.;\sigma^2)$ is no longer non-negative when $\sigma^2<s_c$ so that $\sigma^2$ is not a variance proxy. This contradicts the fact that $\sigma_{\mathrm{opt}}^2<s_c$ is an optimal variance proxy.
\end{proofof}
\begin{proofof}{\Cref{Prop01Zeros}}
    If $g_Y''(.;\sigma^2)$ has no zero on $\mathbb{R}_+^*$, then it is strictly convex on $\mathbb{R}_+^*$ and since $g_Y(0;\sigma^2)=g_Y'(0;\sigma^2)=0$, $g_Y(.;\sigma^2)$ is positive on $\mathbb{R}_+^*$. Similarly, if $g_Y''(.;\sigma^2)$ has a unique zero on $\mathbb{R}_+^*$, then it cannot change sign at this zero, because $g_Y''(0;\sigma^2)=\sigma^2-\operatorname{Var}[Y]\geq 0$ and $\underset{\lambda\to +\infty}{\lim}g_Y''(\lambda;\sigma^2)=\sigma^2>0$. Hence, we conclude similarly that $g_Y(.;\sigma^2)$ is positive on $\mathbb{R}_+^*$. 
\end{proofof}

\begin{lemma}[{Local minimum when $g_Y''(.;\sigma^2)$ has two positive zeros and $\sigma^2>\text{Var}[Y]$}]
\label{LemmaNumber}
%Let us define $\sigma_0^2=\sigma_{\mathrm{opt}}^2 -\frac{1}{2}(\sigma_{\mathrm{opt}}^2-\operatorname{Var}[Y])$ and 
Assume that $M_Y$ is a smooth function and that $M_Y(\lambda)\overset{\lambda\to \pm\infty}{=}o(\lambda^2)$. Moreover, assume that for any $\sigma^2\in \left(\operatorname{Var}[Y],+\infty\right)$, $g_Y''(.;\sigma^2)$ has exactly two positive zeros $(\lambda_1(\sigma),\lambda_2(\sigma))$ such that $0<\lambda_1(\sigma)<\lambda_2(\sigma)$. Then, $g_Y(.;\sigma^2)$ has at most one local minimum  $\lambda_m(\sigma)$ on $\mathbb{R}_+^*$ and it is necessarily located in $\lambda_m(\sigma)\in\left(\lambda_1(\sigma),\lambda_2(\sigma)\right)$.
Consequently, the equations $g_Y(\lambda,\sigma^2)=0=g_Y'(\lambda,\sigma^2)$ have at most one solution on $\mathbb{R}_+^*\times \left(\operatorname{Var}[Y],+\infty\right)$ and we have that:
\begin{itemize}
    \item if they have no solution, then $g_Y(.;\sigma^2)$ is non-negative on $\mathbb{R}_+$ for any $\sigma^2>\operatorname{Var}[Y]$.
    \item if they have one solution $(\lambda_c,\sigma_c^2)$, then $\lambda_c>0$ is necessarily a local minimum and $g_Y(.;\sigma^2)$ is non-negative on $\mathbb{R}_+$ if and only if $\sigma^2\geq \sigma_c^2$. 
\end{itemize}
A similar result is valid on $\mathbb{R}_-^*$. 
\end{lemma}

\begin{proofof}{\Cref{LemmaNumber}} The proof follows from a precise study of variations. Indeed, let us first notice that for any $\sigma^2>\operatorname{Var}[Y]$, we have that $g_Y(.;\sigma^2)$ is locally convex and positive around $\lambda=0$ since $g_Y''(0;\sigma^2)=\sigma^2-\operatorname{Var}[Y]>0$. We also remind that $\underset{\lambda\to \pm \infty}{\lim} g_Y''(\lambda;\sigma^2)=\sigma^2>0$ so that $g_Y(.;\sigma^2)$ is also convex at infinity. Thus, from the assumption that $g_Y''(.;\sigma^2)$ has exactly two zeros $(\lambda_1(\sigma),\lambda_2(\sigma))$ on $\mathbb{R}_+^*$, we conclude that either $g_Y''(.;\sigma^2)$ does not change sign at its zeros and in this case, $g_Y(.;\sigma^2)$ is strictly convex on $\mathbb{R}_+^*$ with $g_Y(0;\sigma^2)=0$ so that it is positive and increasing on $\mathbb{R}_+$, thus there is no solution of $g_Y'(.;\sigma^2)=0$ on $\mathbb{R}_+^*$. Or that $g_Y''(.;\sigma^2)$ is necessarily positive on $(0,\lambda_1(\sigma))$, negative on $(\lambda_1(\sigma),\lambda_2(\sigma))$ and positive on $(\lambda_2(\sigma),+\infty)$. Thus, $g_Y'(.;\sigma^2)$ is increasing on $(0,\lambda_1(\sigma))$ and since $g_Y'(0;\sigma^2)=0$ it is positive on $(0,\lambda_1(\sigma))$. Then, $g_Y'(.;\sigma^2)$ is decreasing on $(\lambda_1(\sigma),\lambda_2(\sigma))$ and then increasing on $(\lambda_2(\sigma),+\infty)$. In particular, $g_Y'(.;\sigma^2)$ admits a unique local minimum at $\lambda=\lambda_2(\sigma)$ on $\mathbb{R}_+^*$. Consequently, if $g_Y'(\lambda_2(\sigma);\sigma^2)\geq 0$, then $g_Y(.;\sigma^2)$ is non-negative on $\mathbb{R}_+$ and thus since $g_Y(0;\sigma^2)=0$, we get that $g_Y(.;\sigma^2)$ is non-negative on $\mathbb{R}_+$. Alternatively, if $g_Y'(\lambda_2(\sigma);\sigma^2)<0$, then $g_Y'(.;\sigma^2)$ has exactly two zeros $\lambda_1(\sigma)<\lambda_l(\sigma)<\lambda_r(\sigma)<\lambda_2(\sigma)$ and it is negative on $(\lambda_l(\sigma),\lambda_r(\sigma))$ and positive on $(0,\lambda_l(\sigma))\cup(\lambda_r(\sigma),+\infty)$. Consequently, $g_Y(.;\sigma^2)$ has exactly one local minimum on $\mathbb{R}_+^*$ denoted $\lambda_m(\sigma)$ which is necessarily located in  $(\lambda_1(\sigma),\lambda_2(\sigma))$.
Next, let us observe that:
\beq \partial_\sigma[g_Y'(\lambda_2(\sigma);\sigma^2)]=g_Y''(\lambda_2(\sigma);\sigma^2) \partial_\sigma[\lambda_2(\sigma)]+ 2\sigma\lambda_2(\sigma)=2\sigma\lambda_2(\sigma)>0.\eeq
Thus, the local minimum of $g_Y'(.;\sigma^2)$ is an increasing function of $\sigma$, so that if it is null for a value $\sigma_1$, then it is positive for $\sigma>\sigma_1$ and negative for $\sigma<\sigma_1$.
Finally, let us observe that
\beq  \partial_\sigma[g_Y(\lambda_m(\sigma);\sigma^2)]=g_Y'(\lambda_m(\sigma);\sigma^2) \partial_\sigma[\lambda_m(\sigma)]+ \sigma\lambda_m(\sigma)^2=\sigma\lambda_m(\sigma)^2>0 .\eeq
so that $\lambda_m(\sigma)$ is an increasing function of $\sigma$. Let us denote $(\lambda_c,\sigma_c^2)\in \mathbb{R}_+^*\times \left(\operatorname{Var}[Y],+\infty\right)$ a solution of $g_Y(\lambda,\sigma^2)=0=g_Y'(\lambda,\sigma^2)$, then for $\sigma<\sigma_c$, we have $g_Y(\lambda_m(\sigma);\sigma^2)<0$ while for $\sigma>\sigma_c$, we have $g_Y(\lambda_m(\sigma);\sigma^2)>0$. Since $g_Y(.;\sigma^2)$ has only a unique local minimum $\lambda_m(\sigma)$ and a local maximum on $\mathbb{R}_+^*$ which is always strictly positive, we conclude that we cannot have another solution of $g_Y(\lambda,\sigma^2)=0=g_Y'(\lambda,\sigma^2)$ with $\lambda\in \mathbb{R}_+^*$. When there is no solution, we have $g_Y(\lambda_m(\sigma);\sigma^2)>0$ so that $g_Y(.;\sigma^2)$ is non-negative on $\mathbb{R}_+^*$. When we have a unique solution $(\lambda_c,\sigma_c^2)$, then $g_Y(\lambda_m(\sigma);\sigma^2)>0$ for $\sigma>\sigma_c$, $g_Y(\lambda_m(\sigma_c);\sigma_c^2)=0$ and $g_Y(\lambda_m(\sigma);\sigma^2)<0$ for $\sigma<\sigma_c$ concluding the proof. 
\end{proofof}

We complement \Cref{LemmaNumber} with another lemma regarding the situation when $\sigma<\operatorname{Var}[Y]$.

\begin{lemma}[{Local minimum when $g_Y''(.;\sigma^2)$ has two positive zeros and $\sigma^2<\operatorname{Var}[Y]$}]\label{LemmaNumber2}
Assume that $M_Y$ is a smooth function and that $M_Y(\lambda)\overset{\lambda\to \pm\infty}{=}o(\lambda^2)$ and that for any $\sigma^2\in \left(0,\operatorname{Var}[Y]\right)$, $g_Y''(.;\sigma^2)$ has exactly two zeros $(\lambda_1(\sigma),\lambda_2(\sigma))$ such that $0<\lambda_1(\sigma)<\lambda_2(\sigma)$. Then, there is no solution to the set of equations $g_Y(\lambda;\sigma^2)=0=g_Y'(\lambda;\sigma^2)$ with $\lambda>0$ and $\sigma^2\in \left(0,\operatorname{Var}[Y]\right)$.\\
A similar result holds on $\mathbb{R}_-^*$.
\end{lemma}

\begin{proofof}{\Cref{LemmaNumber2}} For any $\sigma^2<\operatorname{Var}[Y]$, we have that $g_Y(\lambda\sigma^2)=(\sigma^2-\operatorname{Var}[Y])\lambda^2+o(\lambda^2)$ when $\lambda\to 0_+$. Thus, $g_Y(.;\sigma^2)$ is locally concave and negative around $\lambda=0_+$. Since  $M_Y(\lambda)\overset{\lambda\to \pm\infty}{=}o(\lambda^2)$, we know that $g_Y(.;\sigma^2)$ is locally convex and positive when $\lambda\to +\infty$. Since $g_Y''(.;\sigma^2)$ is assumed to have exactly two positive zeros $\lambda_1(\sigma)<\lambda_2(\sigma)$, we may only have the following cases:
\begin{itemize}
    \item $g''_Y(.;\sigma^2)$ is negative on $(0,\lambda_1(\sigma))$ and changes sign at $\lambda_1(\sigma)$. Thus, it is positive on $(\lambda_1(\sigma),\lambda_2(\sigma))$ and cannot change sign at $\lambda_2(\sigma)$ to remain positive at $+\infty$. Hence, $g''_Y(.;\sigma^2)$ is positive on $(\lambda_1(\sigma),\lambda_2(\sigma))\cup (\lambda_2(\sigma),+\infty)$. Consequently, $g'_Y(.;\sigma^2)$ is decreasing on $(0,\lambda_1(\sigma))$ and increasing on $(\lambda_1(\sigma),+\infty)$. Since $g'_Y(0;\sigma^2)$, we have $g'_Y(\lambda_1(\sigma);\sigma^2)<0$ and since $g'_Y(+\infty;\sigma^2)=+\infty$, $g'_Y$ has only one zero $\lambda_0(\sigma)$ on $\mathbb{R}_+^*$ that satisfies $\lambda_0(\sigma)>\lambda_1(\sigma)$. Moreover, $g'_Y(.;\sigma^2)$ is negative on $(0,\lambda_0(\sigma))$ and positive on $(\lambda_0(\sigma),+\infty)$. Hence, $g_Y(.;\sigma^2)$ is decreasing on $(0,\lambda_0(\sigma))$ and increasing on $(\lambda_1(\sigma),+\infty)$. Since $g_Y(0;\sigma^2)=0$ and $g_Y(+\infty;\sigma^2)=+\infty$, $g_Y$ admits a unique zero $\lambda_*(\sigma)$ on $\mathbb{R}_+^*$ and we have $\lambda_*(\sigma)>\lambda_0(\sigma)$. Hence, there are no simultaneous solutions to $g_Y(\lambda;\sigma^2)=0=g_Y'(\lambda;\sigma^2)$ with $\lambda>0$.
    \item $g''_Y(.;\sigma^2)$ is negative on $(0,\lambda_1(\sigma))$ and does not change sign at $\lambda_1(\sigma)$ so it is negative on $(0,\lambda_2(\sigma))$. In order to be positive at $\lambda\to +\infty$, $g_Y''(.;\sigma^2)$ must change sign at $\lambda=\lambda_2(\sigma)$. Thus, $g_Y'(.;\sigma^2)$ is decreasing on $(0,\lambda_2(\sigma))$ and increasing on $(\lambda_2(\sigma),+\infty)$. Since $g_Y'(0;\sigma^2)=0$ and $g_Y'(+\infty;\sigma^2)=+\infty$, $g_Y'(.;\sigma^2)$ admits a unique zero $\lambda_0(\sigma)$ on $\mathbb{R}_+^*$ and it satisfies $\lambda_0(\sigma)>\lambda_2(\sigma)$. Moreover, $g_Y'(.;\sigma^2)$ is negative on $(0,\lambda_0(\sigma))$ and positive on $(\lambda_0(\sigma),+\infty)$. Since $g_Y(0;\sigma^2)=0$ and $g_Y(+\infty;\sigma^2)=+\infty$, $g_Y(.,\sigma^2)$ is decreasing and negative on $(0,\lambda_0(\sigma))$ and increasing on $(\lambda_0(\sigma),+\infty)$. Hence, $g_Y(.;\sigma^2)$ admits a unique zero $\lambda_*(\sigma)$ on $\mathbb{R}_+^*$ and we have $\lambda_*(\sigma)>\lambda_0(\sigma)$. Hence, there are no simultaneous solutions to $g_Y(\lambda;\sigma^2)=0=g_Y'(\lambda;\sigma^2)$ with $\lambda>0$.
\end{itemize}
\end{proofof}

\section{Proofs for \sectionref{sec:3-mass}}\label{appendix:Appendix_proof_sec_3}

In this section we prove \Cref{TheoRegime1} and \Cref{lemma:lambda_ordering}.

\begin{proofof}{\Cref{TheoRegime1}}
Let us first observe that $\left(2\lambda_0,\sigma_{\mathrm{opt}}^2\right)=\left(-\ln\frac{p_2}{p_1},\sigma_{\mathrm{opt}}^2\right)$ is a solution $(\lambda,\sigma^2)$ of the system of equations
\beqq g_{\sigma,p_1,p_2}(\lambda)=0 \text{ and } g'_{\sigma,p_1,p_2}(\lambda) = 0.\eeqq
Moreover, the case $p_3\leq 4\sqrt{p_1p_2}$ is equivalent to the fact that $\Delta\leq 0$ or $\Delta>0$ with $P$ admitting two strictly negative roots (the sum of roots ($X_1,X_2)$ is $X_1+X_2 = \frac{p_3^2-8p_1p_2}{p_2p_3} < 0$ and the product of roots $X_1 X_2 = \frac{p_1}{p_2}>0$). Thus, the function $N_{p_1,p_2}$ is strictly positive on $\mathbb{R}$. Consequently, $g^{(3)}_{\sigma,p_1,p_2}$ and $u_1(\lambda)$ share the same sign and hence $g_{\sigma,p_1,p_2}^{(3)}$ is strictly negative on $(-\infty,\lambda_0)$ and strictly positive on $(\lambda_{0}, +\infty)$. Furthermore, since $\lim_{\lambda\to \pm \infty} g^{(2)}_{\sigma,p_1,p_2}(\lambda)=\sigma^2$, it follows that $\lambda_0$ is a global minimum of $g^{(2)}_{\sigma,p_1,p_2}$ with value $g^{(2)}_{\sigma,p_1,p_2}(\lambda_{0})=\sigma^2-\frac{2\sqrt{p_1p_2}}{p_3+2\sqrt{p_1p_2}}$. If $\sigma^2\geq\frac{2\sqrt{p_1p_2}}{p_3+2\sqrt{p_1p_2}}$ , then $g^{(2)}_{\sigma,p_1,p_2}$ is non-negative on $\mathbb{R}$ and so $g'_{\sigma,p_1,p_2}$ is a strictly increasing function. Since $g'_{\sigma,p_1,p_2}(0)=0$, $g'_{\sigma,p_1,p_2}$ is negative on $\mathbb{R}_-$ and positive on $\mathbb{R}_+$ and finally $g_{\sigma,p_1,p_2}$ has a global minimum at $\lambda=0$ which is precisely null so it is positive and $\sigma^2$ is a variance proxy. This gives that $\frac{2\sqrt{p_1p_2}}{p_3+2\sqrt{p_1p_2}}$ is an upper bound for the optimal variance proxy. 

On the contrary, if $\sigma^2<\frac{2\sqrt{p_1p_2}}{p_3+2\sqrt{p_1p_2}}$, then $g^{(2)}_{\sigma,p_1,p_2}$ has two distinct zeros $(\lambda_1(\sigma),\lambda_2(\sigma))$ such that $\lambda_1(\sigma)<\lambda_0<\lambda_2(\sigma)\leq 0$. Moreover, since $g_{\sigma,p_1,p_2}''(.;\sigma^2)$ is positive on $\mathbb{R}_+^*$, we get that $g_{\sigma,p_1,p_2}(.;\sigma^2)$ is always positive on $\mathbb{R}_+^*$ for any $\sigma^2\in \left(\operatorname{Var}[Y],\frac{2\sqrt{p_1p_2}}{p_3+2\sqrt{p_1p_2}} \right) $. Since, we have observed that the equations $g_{\sigma,p_1,p_2}(\lambda,\sigma^2)=0=g_{\sigma,p_1,p_2}'(\lambda,\sigma^2)$ admits $\left(2\lambda_0,\frac{2(p_2-p_1)}{\ln(p_2/p_1)}\right)\in \mathbb{R}_-^*\times \left(\operatorname{Var}[Y],+\infty\right)$ as solution, application of \Cref{LemmaMerged} on $\mathbb{R}_-^*$ implies that $g_{\sigma,p_1,p_2}(.;\sigma^2)$ is non-negative on $\mathbb{R}_-$ if and only if $\sigma^2\geq \frac{2(p_2-p_1)}{\ln(p_2/p_1)} $. Thus the optimal variance proxy in this case is $\sigma_{\mathrm{opt}}^2= \frac{2(p_2-p_1)}{\ln(p_2/p_1)}$. 
\end{proofof}

\begin{proofof}{\Cref{lemma:lambda_ordering}}
Let us first observe that we have:
\begin{align*}
    \lambda_{\pm}&= \ln\left(\frac{p_3^2-8p_1p_2 \pm \sqrt{(p_3^2-4p_1p_2)(p_3^2-16p_1p_2)}}{2p_1p_2}\right). 
    % \lambda_{+} &=\ln\left(\frac{p_3^2-8p_1p_2 +  \sqrt{(p_3^2-4p_1p_2)(p_3^2-16p_1p_2)}}{2p_1p_2}\right) .
\end{align*}
We shall denote for compactness $x := p_3^2>0$, $y := p_1p_2>0$ so that we have the condition $x>16y$. Since $\lambda_0<0$, we only need to prove that $\lambda_->0$. We will now prove that $\lambda_{-} > 0$ under the condition $x > 16y>0$. To establish this result, we proceed through a chain of equivalent inequalities, beginning with the definition of $\lambda_{-}$
\begin{align*}
\lambda_{-} > 0 
&\iff \frac{x - 8y - \sqrt{(x-4y)(x-16y)}}{2y} > 1 \\
&\iff x - 8y - \sqrt{(x-4y)(x-16y)} > 2y \\
%&\iff x - 10y > \sqrt{(x-4y)(x-16y)} \\
&\iff (x - 10y)^2 > (x-4y)(x-16y) \quad (\text{because } x > 16y \Rightarrow x-10y > 6y > 0) \\
&\iff x^2 - 20xy + 100y^2 > x^2 - 20xy + 64y^2 \\
&\iff 100y^2 > 64y^2 \\
&\iff 36y^2 > 0 \quad (\text{always valid for  } y \neq 0).
\end{align*}
Thus we get  $\lambda_0 < 0 <\lambda_- <   \lambda_+$ under the condition $p_3^2>16p_1p_2$ ending the proof of the lemma.  
\end{proofof}

\section{Proofs for \sectionref{sec:uniform}}\label{appendix:Appendix_proof_th13}

By linearity of the log-MGF and the scaling property of variance proxies, it is convenient to normalize $X$. Define
\[
Y := \frac{X-b}{a}.
\]
Then $Y$ is uniformly distributed on $\llbracket 1,N\rrbracket$ and
\[
\sigma_{\mathrm{opt}}[X] = |a|\,\sigma_{\mathrm{opt}}[Y].
\]
Hence, without loss of generality, we assume $a=1$ and $b=0$. Under this assumption, the variable $Y$   uniformly distributed on the integer set  $\{1, 2, \ldots, N\}$ with moments:
\beq
  \mu:=\E[Y]=\frac{N+1}{2}, \quad  \sigma^2:=\operatorname{Var}[Y]=\frac{N^2-1}{12}, \quad \operatorname{\kappa_3}[Y] := \E[(Y-\E[Y])^3]=0.
\eeq

By definition, $\sigma>0$ is a variance proxy of $Y$ if and only if \beq \E\left[e^{\lambda Y}\right]= \frac{1}{N}\sum_{k=1}^{N} e^{\lambda k} 
\;\leq\; \exp\!\left(\tfrac{\lambda^2\sigma^2}{2}+\lambda \mu\right), \quad \forall\, \lambda\in \mathbb{R}. \eeq
Equivalently, defining the log-partition function
\[
u(\lambda) := \ln\!\left(\frac{1}{N}\sum_{k=1}^{N} e^{\lambda k}\right),
\]
this condition becomes equivalent to the non-negativity of
\[
g_{\sigma,N}(\lambda) := \tfrac{\lambda^2\sigma^2}{2} - u(\lambda) + \lambda \mu \;\geq 0, 
\qquad \forall \,\lambda \in \mathbb{R}.
\]
It is known that the variance is a universal lower bound for variance proxies. 
Hence in our analysis we consider only $\sigma^2 \geq \operatorname{Var}[Y]= \frac{N^2-1}{12}$.\\

To characterize the optimal variance proxy, it is essential to study the properties of the function $g_{\sigma, N}$. Observe that $g_{\sigma, N}$ is a smooth function of $(\sigma, \lambda)\in \mathbb{R}^2$. Its first three derivatives with respect to $\lambda$ are explicitly computed as:
%\begin{align*}
\beq 
g'_{\sigma,N}(\lambda)   = \lambda \sigma^2 - u'(\lambda) + \mu \,,\,
g^{(2)}_{\sigma,N}(\lambda) = \sigma^2 - u^{(2)}(\lambda)\,,\, 
g^{(3)}_{\sigma,N}(\lambda) = - u^{(3)}(\lambda).
\eeq
%\end{align*}
To analyze the log-partition function $u$, and consequently those of $g_{\sigma, N}$ it is convenient to introduce an auxiliary family of probability distributions. For every real $\lambda$, define the probability distribution $P_\lambda$ on $\{1,\dots,N\}$ by
\[
P_\lambda(k) = \frac{e^{\lambda k}}{\underset{j=1}{\overset{N}{\sum}} e^{\lambda j}}, \qquad \forall\, k\in \llbracket 1, N\rrbracket.
\]
Let $Z_\lambda \sim P_\lambda$ denote the associated random variable.  
%We denote $\E_\lambda[\cdot]$, $\operatorname{Var}_\lambda[\cdot]$, 
%$\kappa_{3,\lambda}[\cdot]$ the expectation, variance and third central moment under $P_\lambda$.  
Note that $Z_0=Y$ coincides with the uniform distribution. This family of probability distributions satisfies several fundamental identities, including symmetry properties and a moment derivative formula.
\paragraph{Symmetry identities.}  
For all $\lambda \in \mathbb{R}$:
%\beq P_{-\lambda}(k) = P_\lambda(N-k+1), \quad  \E_{-\lambda}[Z] = N+1-\E_\lambda[Z], \quad
%\operatorname{Var}_{-\lambda}[Z] = \operatorname{Var}_\lambda[Z]. \eeq
\beq P_{-\lambda}(k) = P_\lambda(N-k+1), \quad  \E[Z_{-\lambda}] = N+1-\E[Z_\lambda], \quad
\operatorname{Var}[Z_{-\lambda}] = \operatorname{Var}[Z_\lambda]. \eeq

\paragraph{Moment derivative identity.}  
For all integers $m \geq 1$:
%\beq \frac{d}{d\lambda} \E_\lambda[Z^m]  = \E_\lambda[Z^{m+1}] - \E_\lambda[Z]\,\E_\lambda[Z^m]. \eeq
\beq \frac{d}{d\lambda} \E[Z_\lambda^m]  = \E[Z_\lambda^{m+1}] - \E[Z_\lambda]\,\E[Z_\lambda^m]. \eeq

\begin{lemma}[Derivatives of log-partition function as moments]\label{lem:log_partition_moments}
The derivatives of $\lambda\mapsto  u(\lambda)$ give the moments of the associated random variables:
\beq
u'(\lambda) = \mathbb{E}[Z_{\lambda}], \quad u^{(2)}(\lambda) = \operatorname{Var}[Z_{\lambda}], \quad u^{(3)}(\lambda) = \mathbb{E}[(Z_{\lambda} - \mathbb{E}[Z_{\lambda}])^3] = \operatorname{\kappa}_{3}[Z_\lambda].
\eeq
\end{lemma}

\begin{proof}
These identities follow directly from the moment derivative formula combined with the expressions for the first moments. Applying the derivative identity recursively yields the expressions for \( u'(\lambda) \), \( u^{(2)}(\lambda) \), and \( u^{(3)}(\lambda) \), which correspond to the mean, variance, and third centered moment of $Z_\lambda$.
\end{proof}

From the symmetry identities, we know that $\operatorname{Var}[Z_{\lambda}]$ is an even function of $\lambda$. This symmetry allows us to restrict our analysis to $\mathbb{R}_+$. Moreover, given the derivative relation
 %\beq \operatorname{\kappa}_{3,\lambda}[Z] = \frac{d}{d\lambda}\operatorname{Var}_{\lambda}[Z]. \eeq
  $ \operatorname{\kappa}_{3}[Z_\lambda] = \frac{d}{d\lambda}\operatorname{Var}[Z_{\lambda}]$ 
along with the evenness of the variance, it follows that the third central moment is an odd function of $\lambda$. The main technical task is then to prove that the third derivative $\lambda \mapsto u^{(3)}(\lambda)$ is negative on $\mathbb{R}_+$.

\begin{lemma}
[Negativity of the third derivative of the log-partition function]\label{lemma:neg_third_deriv}
Let \( N \geq 2 \) be an integer and let \( \lambda > 0 \) be a real number. Then the third derivative of the log-partition function is strictly negative. In other words,
\beq
u^{(3)}(\lambda) = -\frac{N^3 e^{\lambda N}(1 + e^{\lambda N})}{(1 - e^{\lambda N})^3} + \frac{e^{\lambda}(1 + e^{\lambda})}{(1 - e^{\lambda})^3} < 0, \quad \forall\, \lambda\in \mathbb{R^*_+}.
\eeq
Consequently,
\beq
\operatorname{\kappa}_{3}[Z_\lambda] < 0 \text{ for all } \lambda \in \mathbb{R}_+^*.
\eeq
\end{lemma}

\begin{proof}
The explicit expression  of \( u^{(3)}(\lambda)\) follows by direct differentiation of  $u(\lambda)= \ln\left( \frac{1}{N}\underset{k=1}{\overset{N}{\sum}}e^{\lambda k}\right)$ and using the closed-form expression for the geometric sum. Then, consider the auxiliary function \beq f(t) = \frac{t(t + 1)}{(1 - t)^3}, \quad \text{for } t = e^{\lambda} > 1. \eeq
Its derivative is:
\[
f'(t) = \frac{1 + 4t + t^2}{(1 - t)^4}.
\]
Since both the numerator and the denominator are strictly positive for all \( t > 1 \), we conclude that \( f'(t) > 0 \). Hence, \( f \) is strictly increasing on the interval \( (1, \infty) \). Now fix \( \lambda > 0 \), so that \( t = e^{\lambda} > 1 \). Since \( N \geq 2 \), we have \( t^N > t \). By the monotonicity of \( f \), it follows that
\[
f(t^N) > f(t),
\]
which yields the inequality
\[
\frac{e^{\lambda N}(1 + e^{\lambda N})}{(1 - e^{\lambda N})^3} > \frac{e^{\lambda}(1 + e^{\lambda})}{(1 - e^{\lambda})^3}.
\]
Furthermore, since \( N^3 \geq 8 \) for all \( N \geq 2 \), we obtain
\[
\frac{N^3 e^{\lambda N}(1 + e^{\lambda N})}{(1 - e^{\lambda N})^3} > \frac{e^{\lambda}(1 + e^{\lambda})}{(1 - e^{\lambda})^3}.
\]
Therefore, the third derivative of the log-partition function satisfies
\[
u^{(3)}(\lambda) 
= - \frac{N^3 e^{\lambda N}(1 + e^{\lambda N})}{(1 - e^{\lambda N})^3}
+ \frac{e^{\lambda}(1 + e^{\lambda})}{(1 - e^{\lambda})^3}
< 0 \,\,\,,\,\, \forall \,\lambda>0
\]
which completes the proof of the lemma.
\end{proof}
The end of the proof of \Cref{TheoRDisceteDistribution} is now straightforward. 

\paragraph{Variance proxy optimality.} \Cref{lem:log_partition_moments} and
\Cref{lemma:neg_third_deriv} show that
%\[
%u^{(3)}(\lambda)=\kappa_{3,\lambda}[Z] = \frac{d}{d\lambda}\operatorname{Var}_{\lambda}[Z] < 0, \quad \forall\, \lambda > 0.
%\]
\[
u^{(3)}(\lambda)=\kappa_{3}[Z_\lambda] = \frac{d}{d\lambda}\operatorname{Var}[Z_{\lambda}] < 0, \quad \forall\, \lambda > 0.
\]
As \( \operatorname{Var}[Z_{\lambda}] \) is an even function of \( \lambda \), it follows that it is strictly increasing on \( \mathbb{R}_- \), strictly decreasing on \( \mathbb{R}_+ \) and thus achieves a unique global maximum at \( \lambda = 0 \). In particular,
\[
\operatorname{Var}[Z_{\lambda}] \leq \operatorname{Var}[Z_0] = \operatorname{Var}[Y], \quad \forall\, \lambda \in \mathbb{R}.
\]
Hence, for any \( \sigma^2 \geq \operatorname{Var}[Y] \), we have \( g_{\sigma,N}^{(2)}(\lambda) = \sigma^2 - \operatorname{Var}[Z_\lambda] \geq 0 \), with equality if and only if \( \lambda = 0 \). Finally for any \( \sigma^2 \geq \operatorname{Var}[Y] \), given that \( g_{\sigma,N}^{(2)}(\lambda)\) has no solution on $\mathbb{R^*_+} \text{ and } \mathbb{R^*_-}$, it follows from \Cref{Prop01Zeros} that $g_{\sigma,N}$ is non-negative on $\mathbb{R}$. 
%Thus $g_{\sigma,N}$ is nonnegative, and the system
%\[
%g_{\sigma,N}(\lambda,\sigma^2) = 0, 
%\quad 
%\partial_\lambda g_{\sigma,N}(\lambda,\sigma^2) = 0
%\]
%admits a unique solution at $\lambda=0$. 
This shows that $Y$ is strictly sub-Gaussian %with optimal variance proxy \beq \sigma^2_{\mathrm{opt}}[Y] = \operatorname{Var}[Y] = \frac{N^2-1}{12} \eeq
ending the proof.
%Going back to the general case $X = aY+b$, we conclude that
%\beq
%\sigma^2_{\mathrm{opt}}[X] = a^2 \frac{N^2-1}{12}.
%\eeq

\end{document}